\documentclass{amsart}
\usepackage{graphicx} 
\usepackage{hyperref}
\usepackage{amsmath, amsthm, amssymb, amscd}
\usepackage{mathtools}
\usepackage{enumerate}
\usepackage{hyperref}
\usepackage{verbatim}
\usepackage[noadjust]{cite}
\usepackage{esint}
\usepackage{tikz-cd}
\usetikzlibrary{shapes.geometric}
\usetikzlibrary{shapes.misc}
\usepackage[shortlabels]{enumitem}

\title[Conformal Dimension on Laakso-Type Spaces]{Conformal dimension and its attainment on self-similar Laakso-type fractal spaces}

\author{Riku Anttila}
\address[Riku Anttila]{Department of Mathematics and Statistics, University of Jyväskylä, P.O. Box 35, FI-40014 Jyväskylä, Finland}
\email{riku.t.anttila@jyu.fi}

\author{Sylvester Eriksson-Bique}
\address[Sylvester Eriksson-Bique]{Department of Mathematics and Statistics, University of Jyväskylä, P.O. Box 35, FI-40014 Jyväskylä, Finland}
\email{sylvester.d.eriksson-bique@jyu.fi}

\author{Lassi Rainio}
\address[Lassi Rainio]{Department of Mathematics and Statistics, University of Jyväskylä, P.O. Box 35, FI-40014 Jyväskylä, Finland}
\email{lassi.p.j.rainio@student.jyu.fi}

\thanks{The first and third authors were partially supported by the Eemil Aaltonen foundation. The second author was supported by the Research Council of Finland grant 354241. We thank Mathav Murugan and Ryosuke Shimizu for discussions on the attainment problem and conformal walk dimensions. }
\subjclass[2020]{30L10, 53C23, 20F65, 51F99, 28A78}
\keywords{Conformal dimension, Ahlfors regular, attainment problem,  combinatorially Loewner, quasisymmetric mappings, self-similar space, conformal gauge, Loewner spaces, replacement rule, Laakso spaces}
\date{\today}

\newcommand{\diam}{{\rm diam}}
\newcommand{\len}{{\rm len}}
\newcommand{\id}{{\rm id}}

\newcommand{\Mod}{{\rm Mod}}
\newcommand{\Res}{{\rm Res}}
\newcommand{\degr}{{\rm deg}}
\newcommand{\divr}{{\rm div}}
\newcommand{\dist}{{\rm dist}}
\newcommand{\supp}{{\rm supp}}

\newcommand{\AR}{{\rm AR}}
\newcommand{\A}{{\rm A}}

\newcommand{\cN}{{\rm \mathcal{N}}}

\newcommand{\qsequiv}{\sim_{\text{qs}}}
\newcommand{\Adm}{\operatorname{Adm}}

\newtheorem{theorem}[equation]{Theorem}
\newtheorem{lemma}[equation]{Lemma}
\newtheorem{proposition}[equation]{Proposition}
\newtheorem{corollary}[equation]{Corollary}

 \numberwithin{equation}{section}

\theoremstyle{definition}
\newtheorem{definition}[equation]{Definition}

\theoremstyle{remark}
\newtheorem{remark}[equation]{Remark}
\newtheorem{convention}[equation]{Convention}

\newcommand{\N}{\mathbb{N}}
\newcommand{\R}{\mathbb{R}}
\newcommand{\cG}{\mathcal{G}}
\newcommand{\cE}{\mathcal{E}}

\newcommand{\cJ}{\mathcal{J}}
\newcommand{\cH}{\mathcal{H}}
\newcommand{\cI}{\mathcal{I}}
\newcommand{\cM}{\mathcal{M}}

\definecolor{CB_SAFE_A}{HTML}{648fff}
\definecolor{CB_SAFE_B}{HTML}{785ef0}
\definecolor{CB_SAFE_C}{HTML}{dc267f}

\makeatletter
\let\c@equation\c@figure
\makeatother

\begin{document}

\begin{abstract}
A general construction of Laakso-type fractal spaces was recently introduced by the first two authors. In this paper, we establish a simple condition characterizing when the Ahlfors regular conformal dimension of a symmetric Laakso-type fractal space is attained. The attaining metrics are constructed explicitly. This gives new examples of attainment and clarifies the possible obstructions.
\end{abstract}

\maketitle

\section{Introduction}

\subsection{Overview}

\emph{Conformal dimension} is a quasisymmetric invariant introduced by Pansu \cite{P89}. Given a metric space, its \emph{(Ahlfors regular) conformal dimension} is the infimum of the Hausdorff dimensions of all Ahlfors regular metric spaces quasisymmetrically equivalent to it; see Subsection \ref{Subsection: Conformal dimension} for detailed definition. While the value of this infimum can often be estimated numerically \cite{Carrasco}, in most cases it remains unknown whether the infimum is realized --- that is, whether the conformal dimension of the metric space is \emph{attained}. This question is known as the \emph{attainment problem}.

A better understanding of attainment could provide answers to several open problems. For example, Bonk and Kleiner \cite{BK05} showed that Cannon's conjecture, a long-standing conjecture in geometric group theory, hinges on addressing the attainment problem for the boundaries of certain hyperbolic groups. A version of the attainment problem is also equivalent to the so-called Kleiner's conjecture \cite[Conjecture 7.5]{KleinerICM} on metric spaces satisfying the \emph{combinatorial Loewner property}. While counterexamples to this conjecture were recently found by the first two named authors \cite{anttila2024constructions}, the problem of characterizing attainment was left open.

In this paper, we fully resolve the attainment problem for a class of metric spaces called \emph{symmetric Laakso-type fractal spaces}. These include the counterexamples considered in \cite{anttila2024constructions} and the well-known fractals introduced by Laakso \cite{La00, LangPlaut,Laakso}. These fractals also give simpler models for boundaries of hyperbolic buildings and Fuchsian buildings considered in \cite{BourdonPI}. This significantly enlarges the class for which the attainment problem is understood. While interesting examples, such as the Sierp\'nski carpet, remain open, the class of symmetric Laakso-type fractal spaces is a natural first step towards understanding the attainment problem in a more general context.

\subsection{Main Theorem} 
Each Laakso-type fractal space arises as a certain \emph{limit space} of a sequence of graph approximations generated by an \emph{iterated graph system (IGS)}. A detailed description of this construction is provided in Section \ref{Section: Construction of Laakso-type fractal spaces}; here we provide a brief outline of the procedure.

    \begin{figure}[!ht]
    \input{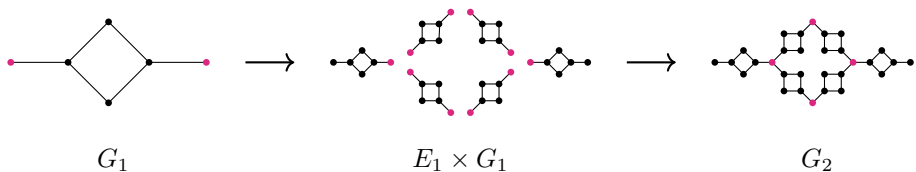}
    \caption{First iteration of the procedure that produces a sequence of graphs converging to a limit space known as the Laakso diamond.}
    \label{Figure: Construction of Laakso Diamond}
    \end{figure}

The construction of a Laakso-type fractal begins with a \emph{generator} graph \(G_1\). One then produces a sequence of graphs \((G_n)_{n\in\N}\) inductively as follows. Given \(G_n\), each edge of \(G_n\) is first replaced with a copy of \(G_1\). The copies meeting at junctions are then glued together by identifying vertices according to a fixed gluing rule, thereby producing \(G_{n+1}\). After equipping each graph \(G_n\) with a suitable semi-metric \(d_n\), one obtains the Laakso-type fractal space \((X,d_{L_*})\) as the Gromov–Hausdorff limit of the sequence \((G_n, d_n)_{n\in\N}\). The limit metric \(d_{L_*}\) is called the intrinsic metric; see Definition \ref{def:intrinsicmetric}. While the construction of the metric is slightly technical, in many cases, this metric is comparable to a re-scaled limit of weighted path metrics on $G_n$; see Remark \ref{rmk:quasiconvex}. We call a Laakso-type fractal space \emph{symmetric}, if $G_1$ has a certain mirror symmetry; see the next subsection for more discussion.

    \begin{figure}[!ht]
    \includegraphics{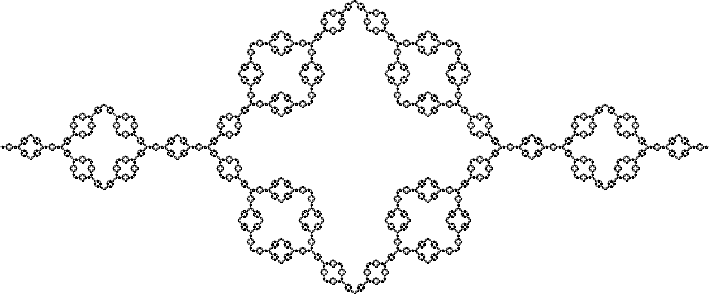}
    \caption{The Laakso diamond.}
    \label{Figure: Laakso Diamond}
    \end{figure}
    
Our main results demonstrate that the quasiconformal geometry of a symmetric Laakso-type space can be understood by solving a discrete modulus problem on the generator. In particular, the challenging attainment problem can be resolved in this setting by determining whether the generator contains a so-called \emph{removable edge}. 

    \begin{theorem}
    \label{Main Theorem: Characterization of Attainment}
    A symmetric Laakso-type fractal space \((X,d_{L_*})\) attains its Ahlfors regular conformal dimension if and only if it has no removable edges.
    \end{theorem}

We now proceed to describe the terminology appearing in this theorem and to sketch the techniques used in the proof.

\subsection{The construction of metrics using admissible functions}

The generator of a symmetric Laakso-type fractal space contains two \emph{gluing sets}, \(I_-,I_+\subset V_1\), and admits a graph isometry \(\eta\colon G_1\to G_1\), for which \(\eta(I_-)=I_+\). One may think of the sets \(I_-\) and \(I_+\) as opposite ends of the generator; at the gluing stage, copies of the generator are always glued end-to-end.

    \begin{figure}[!ht]
    \begin{tikzpicture}[scale=0.8]

\draw (-4.5,1)--(-3.5,0)--(-2.5,1);
\draw (-4.5,-1)--(-3.5,0)--(-2.5,-1);
\draw[->,thick] (-1.75,0)--(-0.5,0);
\draw[->,thick] (3,0)--(4.25,0);

\draw (5,1.25)--(5.5,0.5)--(6,0.25);
\draw (5,0.75)--(5.5,0.5)--(6,-.25);
\draw (5,-1.25)--(5.5,-0.5)--(6,0.25);
\draw (5,-0.75)--(5.5,-0.5)--(6,-0.25);

\draw (6,0.25)--(6.5,0.5)--(7,1.25);
\draw (6,-0.25)--(6.5,0.5)--(7,0.75);
\draw (6,0.25)--(6.5,-0.5)--(7,-1.25);
\draw (6,-0.25)--(6.5,-0.5)--(7,-0.75);

\node at (-4.5,1) [color=CB_SAFE_A,circle,fill,inner sep=1pt]{};
\node at (-3.5,0) [circle,fill,inner sep=1pt]{};
\node at (-2.5,1) [color=CB_SAFE_A,circle,fill,inner sep=1pt]{};
\node at (-2.5,-1) [color=CB_SAFE_C,circle,fill,inner sep=1pt]{};
\node at (-4.5,-1) [color=CB_SAFE_C,circle,fill,inner sep=1pt]{};

\draw (0,1.5)--(0.5,0.875)--(1,0.25);
\draw (0,1)--(0.5,0.875)--(1,0.75);

\node at (0,1.5) [circle,fill,inner sep=1pt]{};
\node at (0,1) [circle,fill,inner sep=1pt]{};
\node at (0.5,0.875) [circle,fill,inner sep=1pt]{};
\node at (1,0.25) [color=CB_SAFE_C,circle,fill,inner sep=1pt]{};
\node at (1,0.75) [color=CB_SAFE_A,circle,fill,inner sep=1pt]{};

\draw (0,-1.5)--(0.5,-0.875)--(1,-0.25);
\draw (0,-1)--(0.5,-0.875)--(1,-0.75);

\node at (0,-1) [circle,fill,inner sep=1pt]{};
\node at (0,-1.5) [circle,fill,inner sep=1pt]{};
\node at (0.5,-0.875) [circle,fill,inner sep=1pt]{};
\node at (1,-0.75) [color=CB_SAFE_C,circle,fill,inner sep=1pt]{};
\node at (1,-0.25) [color=CB_SAFE_A,circle,fill,inner sep=1pt]{};

\draw (2.5,1.5)--(2,0.875)--(1.5,0.25);
\draw (2.5,1)--(2,0.875)--(1.5,0.75);

\node at (2.5,1.5) [circle,fill,inner sep=1pt]{};
\node at (2.5,1) [circle,fill,inner sep=1pt]{};
\node at (2,0.875) [circle,fill,inner sep=1pt]{};
\node at (1.5,0.25) [color=CB_SAFE_C,circle,fill,inner sep=1pt]{};
\node at (1.5,0.75) [color=CB_SAFE_A,circle,fill,inner sep=1pt]{};

\draw (2.5,-1.5)--(2,-0.875)--(1.5,-0.25);
\draw (2.5,-1)--(2,-0.875)--(1.5,-0.75);

\node at (2.5,-1) [circle,fill,inner sep=1pt]{};
\node at (2.5,-1.5) [circle,fill,inner sep=1pt]{};
\node at (2,-0.875) [circle,fill,inner sep=1pt]{};
\node at (1.5,-0.75) [color=CB_SAFE_C,circle,fill,inner sep=1pt]{};
\node at (1.5,-0.25) [color=CB_SAFE_A,circle,fill,inner sep=1pt]{};

%\node at (2.5,0.5) [circle,fill,inner sep=1pt]{};
%\node at (2.5,-0.5) [circle,fill,inner sep=1pt]{};
%\node at (3,1.25) [circle,fill,inner sep=1pt]{};
%\node at (3,0.75) [circle,fill,inner sep=1pt]{};
%\node at (3,-0.75) [circle,fill,inner sep=1pt]{};
%\node at (3,-1.25) [circle,fill,inner sep=1pt]{};

\node at (5,1.25) [circle,fill,inner sep=1pt]{};
\node at (5.5,0.5) [circle,fill,inner sep=1pt]{};
\node at (5,0.75) [circle,fill,inner sep=1pt]{};
\node at (5.5,-0.5) [circle,fill,inner sep=1pt]{};
\node at (5,-0.75) [circle,fill,inner sep=1pt]{};
\node at (5,-1.25) [circle,fill,inner sep=1pt]{};
\node at (6,0.25) [color=CB_SAFE_A,circle,fill,inner sep=1pt]{};
\node at (6,-0.25) [color=CB_SAFE_C,circle,fill,inner sep=1pt]{};

\node at (6.5,0.5) [circle,fill,inner sep=1pt]{};
\node at (6.5,-0.5) [circle,fill,inner sep=1pt]{};
\node at (7,1.25) [circle,fill,inner sep=1pt]{};
\node at (7,0.75) [circle,fill,inner sep=1pt]{};
\node at (7,-0.75) [circle,fill,inner sep=1pt]{};
\node at (7,-1.25) [circle,fill,inner sep=1pt]{};

%\node at (-4.5,0.1) {\(I^-\)};
%\node at (-2.3,0.1) {\(I^+\)};
\node at (-4.5,1.3) {\(v^-\)};
\node at (-2.2,1.3) {\(v^+\)};
\node at (-4.5,-1.3) {\(w^-\)};
\node at (-2.2,-1.3) {\(w^+\)};
\node at (-3.5,0.3) {\(u\)};
\node at (1.25,-2) {\(G_1\times E_1\)};
\node at (-3.4,-2) {\(G_1\)};
\node at (6.1,-2) {\(G_2\)};
\end{tikzpicture}
    \caption{The generator \(G_1\) of a symmetric Laakso-type fractal space with gluing sets \(I_-=\{v^-,w^-\}\) and \(I_+=\{v^+,w^+\}\). \(G_2\) is produced by replacing each edge of \(G_1\) with a copy of itself and gluing the copies along the gluing sets at the junction \(u\). The symmetry exchanges \(v^+\) with \(v^-\) and \(w^+\) with \(w^-\).}
    \label{fig:placeholder1}
    \end{figure}
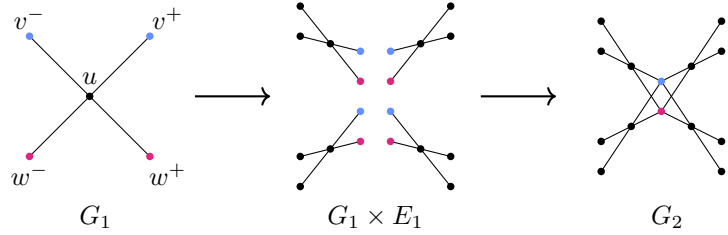

As our main tool, we construct a family of quasisymmetrically equivalent metrics on the limit space, including the intrinsic metric \(d_{L_*}\). Inspired by \cite{Carrasco, KeithLD,Kigamiweighted}, each metric in the family is given by an \emph{admissible density function} on the generator. Similarly as in these references, we denote the family of edge paths from \(I_-\) to \(I_+\) by \(\Theta^{(1)}\), and we say that a function \(\rho\colon E_1\to [0,1]\) is a \(\Theta^{(1)}\)-admissible density if \(\sum_{e\in \theta}\rho(e) \geq 1\) for every \(\theta\in\Theta^{(1)}\). A density \(\rho\) is \emph{symmetric} if \(\rho(\eta(e))=\rho(e)\). The following theorem describes the family of metrics associated to such densities.
 
\begin{theorem}\label{Main Theorem: Construction of Metrics}
    For every symmetric \(\Theta^{(1)}\)-admissible density \(\rho\colon E_1\to (0,1)\), there exists a metric \(d_\rho\) on \(X\) so that \((X,d_\rho)\) is \(Q\)-Ahlfors regular, where \(Q\geq 1\) is the unique value that satisfies the equation:
    \begin{equation}\label{eq:ARregexp}
        \sum_{e\in E_1}\rho(e)^Q=1.
    \end{equation}
    Moreover, for any two such densities \(\rho\) and \(\nu\), the spaces \((X,d_\rho)\) and \((X,d_\nu)\) are quasisymmetrically equivalent.
\end{theorem}

As a result of Theorem \ref{Main Theorem: Construction of Metrics}, we obtain upper bounds for the conformal dimension. Obtaining lower bounds is generally much harder. Using an approach motivated by a series of works \cite{P89,TQuasi,Tyson,KL,Carrasco}, we obtain the conformal dimension as a critical exponent of a discrete modulus problem.  For any \(p\geq 1\), the \(p\)-edge modulus of \(\Theta^{(1)}\) is the value
    
    \[
        \Mod_p(\Theta^{(1)},G_1):=\inf\bigg\{\sum_{e\in E_1}\rho(e)^p\ \bigg\vert\ \rho\colon E_1\to [0,1]\text{ is }\Theta^{(1)}\text{-admissible}\bigg\}.
    \]

    \begin{theorem}
    \label{Main Theorem: Value of Conformal Dimension}
    The Ahlfors regular conformal dimension of \((X,d_{L_*})\) is the unique value \(Q_*\geq 1\) satisfying \(\Mod_{Q_*}(\Theta^{(1)},G_1)=1\).
    \end{theorem}

In contrast to a similar characterization of conformal dimension in \cite{Carrasco}, our special case offers some key simplifications. While one generally needs to consider the asymptotics of certain modulus problems on a sequence graphs with exponentially growing complexity, in our setting we only need to solve one modulus problem for a single graph. Compare for example to the involved numerical estimation of the conformal dimension of the Sierpi\'nski carpet in \cite{Kwapisz}.

It follows from convex optimization theory that whenever \(Q_*>1\), there exists a unique admissible density \(\rho^*\colon E_1\to [0,1)\) that is \(Q_*\)-optimal, in the sense that

    \[
        \sum_{e\in E_1}\rho^*(e)^{Q_*}=\Mod_{Q_*}(\Theta^{(1)},G_1)=1.
    \]   
Due to the symmetry of the generator, the optimal density \(\rho^*\) is also symmetric. It then follows from Theorem \ref{Main Theorem: Construction of Metrics} and Theorem \ref{Main Theorem: Value of Conformal Dimension} that the conformal dimension is attained if \(\rho^*(e)>0\) for all \(e\in E_1\). Our main theorem, Theorem \ref{Main Theorem: Characterization of Attainment}, asserts that this condition is also necessary.

    \begin{figure}[!ht]
    \begin{tikzpicture}
    \node at (-5,1) [circle,fill,inner sep=1pt]{};
    \node at (-5,-1) [circle,fill,inner sep=1pt]{};
    \node at (-4,0) [circle,fill,inner sep=1pt]{};
    \node at (-3,-1) [circle,fill,inner sep=1pt]{};
    \node at (-3,1) [circle,fill,inner sep=1pt]{};
    \node at (-2,0) [circle,fill,inner sep=1pt]{};
    \node at (-1,1) [circle,fill,inner sep=1pt]{};
    \node at (-1,-1) [circle,fill,inner sep=1pt]{};

    \draw (-5,1) -- (-4,0) -- (-3,1) -- (-2,0) -- (-1,1);
    \draw (-5,-1) -- (-4,0) -- (-3,-1) -- (-2,0) -- (-1,-1);
    \draw[-] (-3,1) to (-3,-1);

    \node at (1,1) [circle,fill,inner sep=1pt]{};
    \node at (1,0) [circle,fill,inner sep=1pt]{};
    \node at (1,-1) [circle,fill,inner sep=1pt]{};
    \node at (2,0) [circle,fill,inner sep=1pt]{};
    \node at (4,0) [circle,fill,inner sep=1pt]{};
    \node at (5,1) [circle,fill,inner sep=1pt]{};
    \node at (5,0) [circle,fill,inner sep=1pt]{};
    \node at (5,-1) [circle,fill,inner sep=1pt]{};

    \draw (1,1) -- (2,0);
    \draw (4,0) -- (5,1);
    \draw (1,0) -- (2,0);
    \draw (1,-1) -- (2,0);
    \draw (4,0) -- (5,-1);
    \draw (4,0) -- (5,0);
    \draw[-] (2,0) to [bend left = 5em] (4,0);
    \draw[-] (2,0) to [bend right = 5em] (4,0);
    
\end{tikzpicture}
    \caption{Two generators for symmetric Laakso-type spaces. The one on the left includes a vertical edge positioned in the middle of the graph. This edge is invariant under the symmetry, and thus must be removable. This example was the first one to be found to contradict Kleiner's conjecture, see \cite{anttila2024constructions}. The graph on the right does not have a removable edge, thereby giving an example where attainment occurs and the conjecture holds by Theorem \ref{Main Theorem: Characterization of Attainment} and Corollary \ref{cor:Loewner}. This graph is obtained by taking two parallel edges connecting two vertices in the middle, and adjoining three edges to each side in order to connect the boundary.}
    \label{fig:placeholder}
    \end{figure}
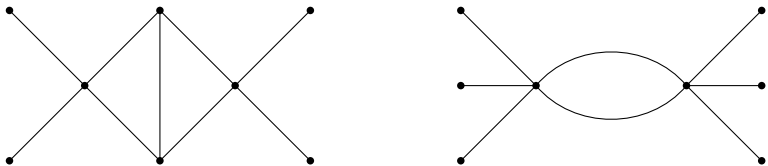

Any edge \(e^*\in E_1\) for which \(\rho^*(e^*)=0\) is called a \emph{removable edge}. Intuitively, the removal of such an edge leaves the value of the discrete modulus problem, and thus also the conformal dimension, unchanged. Using this observation, we prove Theorem \ref{Main Theorem: Characterization of Attainment} by showing that the presence of a removable edge implies the existence of a porous subset \(\widehat{X}\subset X\) with the same conformal dimension, which is a known obstruction to attainment. 

\subsection{Kleiner's conjecture and Sobolev spaces}

One of the original motivations for the study of Laakso-type spaces was Kleiner's conjecture on the attainment of conformal dimension for a certain family of self-similar spaces.
See Section \ref{Section: Approximate self-similarity and the Combinatorial Loewner Property} for a precise definition and detailed discussion. In \cite{anttila2024constructions}, the first two authors constructed counter-examples to this conjecture by exhibiting Laakso-type spaces having removable edges. Our results here show that there is also large family of Laakso type spaces that yield non-trivial positive examples for this conjecture; see Figure \ref{fig:placeholder}, Corollary \ref{cor:Loewner} and Theorem \ref{thm:CLP}. In these cases, the attaining metrics have an explicit description, but they lack a simple formula in terms of e.g. snowflaking. The attaining metrics also seem interesting in their own right, since they satisfy the Loewner property.

Another motivation for this work comes from the study of Sobolev spaces on fractals defined by re-scaled energies, an approach spearheaded by \cite{KuzuokaZhou} and further developed for all $p>1$ in \cite{kigami,shimizu,murugan2023first}. One can see hints of such constructions also in \cite{P89, BS}. For Laakso type spaces these spaces and their properties were studied in \cite{ASEBShimizu}; see therein also for an introduction on Sobolev spaces on fractals. Since \cite{murugan2023first}, it has been known that Sobolev spaces and their energy constructions are closely related to attainment problems; see \cite[Remark 8.7]{ASEBShimizu} for further discussion. The functions $f$ in the Sobolev space are associated with a natural energy measure $\Gamma_p\langle f\rangle$, and a measure class of minimal energy dominating measures $\Lambda$ that are minimal measures s.t. absolute continuity of energy measures, $\Gamma_p\langle f\rangle \ll \Lambda$, holds for all Sobolev functions $f$. In \cite{murugan2023first, MN}, it was suggested that the attainment problem is closely related to the existence of a doubling energy dominating measure. For Laakso type spaces, our work shows that the relationship between these problems becomes an equivalence. Indeed, for Laakso type spaces, the existence of a doubling minimal energy dominant measure is characterized in terms of removable edges \cite[Theorem 8.4]{ASEBShimizu}. This is the same condition as we obtained for attainment. Thus, for these spaces, attainment is equivalent to the existence of a doubling minimal energy dominant measure. This closely ties together the energy based approach to attainment in \cite{murugan2023first,ASEBShimizu} and the explicit metric constructions used in this paper. These remarks prove the following alternate characterization for attainment.

\begin{theorem}
    A symmetric Laakso-type fractal space attains its Ahlfors regular conformal dimension if and only if there is a doubling minimal energy dominant measure for the Sobolev space constructed in \cite{ASEBShimizu}.
\end{theorem}

For general spaces such a simple characterization is not expected to hold. See for example \cite{MN}, where a related notion of conformal walk dimension is discussed. There, the Sierpi\'nski tetrahedron is shown to be a space with a doubling energy dominant measure, but not attaining its conformal walk dimension. The case of conformal dimension (greater than one) is analogous, but different, and seems to require the study of more complicated examples. For general spaces and for conformal dimension instead of conformal walk dimension, the attainment problem seems to depend on a much more delicate study of the associated energy measures. To indicate how difficult these problems are, we note that in many cases, such as the Sierpi\'nski carpet, it is not presently even known if a doubling minimal energy dominant measure exists. Much more work is needed to understand these points fully. 

\subsection{Outline of the paper}

In Section \ref{Section: Preliminaries}, we present some terminology and preliminaries. The construction of Laakso-type fractal spaces is presented in Section \ref{Section: Construction of Laakso-type fractal spaces}. In Section \ref{Section: Construction of Metrics}, we present a construction of metrics from admissible densities and prove Theorem \ref{Main Theorem: Construction of Metrics}. In Section \ref{Section: Moduli and Conformal Dimension}, we characterize attainment via removable edges, proving Theorems \ref{Main Theorem: Value of Conformal Dimension} and \ref{Main Theorem: Characterization of Attainment}. In Section \ref{Section: Approximate self-similarity and the Combinatorial Loewner Property}, we verify the approximate self-similarity and the combinatorial Loewner property for symmetric Laakso-type fractal spaces.

\section{Preliminaries}\label{Section: Preliminaries}

In this section, we present the relevant terminology and preliminaries from metric geometry and graph theory. For further details, we refer to \cite{He,bbi,MT,NakamuraYamasakiDuality}.

\subsection{Metric geometry}\label{Subsection: Conformal dimension}

Given a metric space \((X,d)\), we define for every \(x\in X\) and \(r>0\) the open ball of center \(x\) and radius \(r\) as the set
    \[
        B_d(x,r):=\{y\in X\mid d(x,y)<r\}.
    \]
Given two nonempty sets \(A,B\subseteq X\), we define the diameter of \(A\) and the distance between \(A\) and \(B\) as
    \[
        \diam(A,d):=\sup_{x,y\in A} d(x,y)
        \quad\text{and}\quad\dist(A,B,d):=\inf_{a\in A,b\in B} d(a,b),
    \]
respectively. When the metric is clear from the context, we omit it from the notation and simply write \(B(x,r)\), \(\diam(A)\) and \(\dist(A,B)\). For $Q > 0$, the  \emph{$Q$-dimensional Hausdorff measure} of a subset \(A\subset X\) is
\[
    \mathcal{H}^Q(A):=\lim_{\delta\to 0}\mathcal{H}_\delta^Q(A),
\]
where for all \(\delta>0\) the \emph{\(Q\)-dimensional Hausdorff \(\delta\)-content} of \(A\) is defined as
\[
    \mathcal{H}^Q_\delta(A):=\inf\bigg\{\sum_{i=1}^\infty \diam(A_i)^Q\; \bigg| \;A\subset\bigcup_{i=1}^\infty A_i,\;\diam(A_i)\leq\delta \bigg\}.
\]
The \emph{Hausdorff dimension} of \((X,d)\) is the value
    \[
        \dim_{\mathrm{H}}(X,d):=\inf\{ Q>0\mid \mathcal{H}^Q(X)=0\}.
    \] 

\begin{definition}\label{def:AR}

    We say that a metric space $(X,d)$ is \emph{\(Q\)-Ahlfors regular} for \(Q>0\) if there is a Radon measure \(\mu\) on $(X,d)$ and a constant $C \geq 1$ such that
    \begin{equation}\label{AReq}
    C^{-1} r^Q \leq \mu(B(x,r))\leq Cr^Q \text{ for all } x \in X \text{ and } r \in (0,2\diam(X)].
    \end{equation}
    If \((X,d)\) is \(Q\)-Ahlfors regular for some $Q > 0$, we say that $(X,d)$ is Ahlfors regular.
\end{definition}

We now review two standard properties of Ahlfors regular metric spaces. First, it follows from a volume counting argument that Ahlfors regular metric spaces satisfy the \emph{metric doubling property}. This means that there is $N \in \N$ such that, for every \(x\in X\) and \(r>0\), there are $x_1,\dots, x_N \in X$ satisfying
    \[
        B(x,2r)\subset\bigcup_{i=1}^N B(x_i,r).
    \]
Second, if a metric space is $Q$-Ahlfors regular for $Q > 0$, then its Hausdorff dimension is equal to $Q$ \cite[Pages 61-62]{He}. Moreover, $\mu=\cH^Q$ satisfies \eqref{AReq}.

\subsection{Conformal dimension} We now present the terminology and some results related to conformal dimension; see the monograph \cite{MT} by Mackay and Tyson for a comprehensive introduction.
Given a homeomorphism $\eta\colon[0,\infty)\to[0,\infty)$, we say that a homeomorphism between metric spaces \(f\colon (X,d_X) \to (Y,d_Y)\) is an \emph{\(\eta\)-quasisymmetry} if for all triples \(x,y,z\in X\) with \(x\neq z\),
    \[
    \frac{d_Y(f(x),f(y))}{d_Y(f(x),f(z))}\leq\eta\bigg(\frac{d_X(x,y)}{d_X(x,z)}\bigg).
    \]
    If $f : (X,d_X) \to (Y,d_Y)$ is an $\eta$-quasisymmetry for some $\eta$, we say that $f$ is a \emph{quasisymmetry}. We say that two metrics \(d\) and \(\delta\) on a set  \(X\) are \emph{quasisymmetrically equivalent} and denote \((X,d)\qsequiv(X,\delta)\), if the identity map \(\id\colon (X,d)\to(X,\delta)\) is a quasisymmetry. The Ahlfors regular conformal gauge of \((X,d)\) is the collection
    \[
    \mathcal{G}_{\AR}(X,d):=
        \left\{\delta\ \middle\vert \begin{array}{l}
            \ \delta\text{ is a metric on }X,\ (X,d)\qsequiv(X,\delta)\text{ and}\\
            (X,\delta)\text{ is }Q\text{-Ahlfors regular for some }Q\geq1
        \end{array} \right\},
    \]
and the Ahlfors regular conformal dimension of \((X,d)\) is the value
    \[
    \dim_{\text{AR}}(X,d)=\inf\{\dim_{\mathrm{H}}(X,\delta)\mid \delta \in \mathcal{G}_{\text{AR}}(X,d) \}.
    \]

We note that the Ahlfors regular conformal dimension is not monotone with respect to inclusion. For example, the unit square \(X=([0,1]^2,\lvert \cdot \rvert)\) has conformal dimension $\dim_{\AR}(X)=2$, but there are no Ahlfors regular metrics on the subset $Y= \{1/n : n\in \N\}\cup\{0\}\subset X$, and hence $\dim_{\AR}(Y)=\infty$. Thus, without some additional assumptions on the subset, it is not automatic that $\dim_{\AR}(Y)\leq \dim_{\AR}(X)$ whenever $Y\subset X$. For this reason, we consider the Assouad dimension.

\begin{definition}
    Let \((X,d)\) be a metric space. For all \(A\subset X\), denote
    \[
        N_r(A):=\inf\left\{n\in\N\ \middle\vert\ \exists(x_i)_{i=1}^n\subset X\text{ such that  } A\subset\bigcup_{i=1}^nB(x_i,r)\right\}.
    \]
The \emph{Assouad dimension} of a metric space \(X\) is defined as the value
    \[
        \dim_\A(X):= \inf\left\{ \alpha>0\ \middle\vert\ \begin{array}{l}
        \exists C\geq 1 \text{ such that for all }0<r<R\\
        \text{and } x\in X,\ N_r(B(x,R))\leq C(R/r)^\alpha
        \end{array}\right\}.
    \]
\end{definition}

The Assouad dimension is monotone; if \(Y\subset X\), then \(\dim_\A(Y)\leq \dim_\A(X)\). This follows quite directly from the definition. Moreover, if a metric space \((X,d)\) is Ahlfors regular, then it follows from a standard volume counting argument that \(\dim_\A(X,d)=\dim_{\mathrm{H}}(X,d)\). We make use of these facts by considering subsets that are complete and \emph{uniformly perfect}.

\begin{definition}
 A metric space \((X,d)\) is uniformly perfect if there exists a constant \(C\geq 1\) so that 
        \(
            B(x,r)\setminus B(x,r/C)\neq \varnothing
        \)
    for every \(x\in X\) and \(r\in(0,\diam(X,d))\).    
\end{definition}

We remark that a uniformly perfect metric space lying inside an Ahlfors regular metric space supports an Ahlfors regular metric. That is, $\cG_{\AR}(X,d)\neq \emptyset$, and thus $\dim_{\AR}(X)<\infty$. In fact, this is an equivalence because every Ahlfors regular space is uniformly perfect.

\begin{lemma}[{\cite[Proposition 14.14]{He}}]\label{Lemma: Confdim bound by Assouad dimension}
    Let \((X,d)\) be an Ahlfors regular metric space. If \(Y\subset X\) is complete and uniformly perfect, then  \(\dim_\AR(Y,d\vert_Y)\leq \dim_{\mathrm{A}}(X,d)\). 
\end{lemma}

It follows quite directly that monotonicity of conformal dimension, with respect to inclusion, holds for subsets that are complete and uniformly perfect.

\begin{corollary}\label{Corollary: Confdim bound for uniformly perfect subset}
    Let \((X,d)\) be a metric space with \(\dim_\AR(X,d)<\infty\). If a subset \(Y\subset (X,d)\) is complete and uniformly perfect, then \(\dim_\AR(Y,d\vert_Y)\leq \dim_\AR(X,d)\).
\end{corollary}
    \begin{proof}
        Let \(\delta\in \cG_\AR(X,d)\). As uniform perfectness is preserved by quasisymmetric maps (\cite[Exercise 11.2]{He}), the space \((Y,\delta\vert_Y)\) is complete and uniformly perfect. Consequently, we have by Lemma \ref{Lemma: Confdim bound by Assouad dimension} that
            \[
            \dim_\AR(Y,d\vert_Y)=\dim_\AR(Y,\delta\vert_Y)\leq \dim_\A(Y,\delta\vert_Y)\leq \dim_\A(X,\delta)=\dim_{\mathrm{H}}(X,\delta).
            \]
        Since \(\delta\in \cG_\AR(X,d)\) was arbitrary, it follows that 
            \[
            \dim_\AR(Y,d\vert_Y)\leq \inf_{\delta\in\cG_\AR(X,d)}\dim_{\mathrm{H}}(X,d)=\dim_\AR(X,d).
            \]  
    \end{proof}

A known obstruction to attainment, is the existence of suitable \emph{porous} subsets. We say that a subset \(Y\) of a metric space \(X\) is porous, if there exists a constant \(\delta>0\) such that for all \(y\in Y\) and \(r\in(0,\diam(X))\) there exists a point \(x\in X\) so that \(B(x,\delta r)\subset B(y,r)\setminus Y\).

\begin{proposition}[{\cite[Proposition 2.9]{anttila2024constructions}}]
\label{Proposition: Porosity implies non-attainment}
    Let \((X,d)\) be a compact Ahlfors regular metric space and let \(Y\subset X\) be a porous subset. If \(\dim_\AR(Y,d\vert_Y)=\dim_\AR(X,d)\), then the Ahlfors regular conformal dimension of \((X,d)\) is not attained.
\end{proposition}

We remark, that the proof for \cite[Proposition 2.9]{anttila2024constructions} contains a small gap, since  the uniform perfectness of the subset $Y$ is not discussed. This is implied by the assumption of \(\dim_\AR(Y,d\vert_Y)=\dim_\AR(X,d)\), together with the implied assumption that \(\dim_\AR(X,d)<\infty\); see the paragraph before Lemma \ref{Lemma: Confdim bound by Assouad dimension}. In fact, the same proof and statement appears in \cite[Proposition 8.3]{CM}, where the same omission occurs.

\subsection{Graphs and Modulus} A \emph{graph} is an ordered triple \(G:=(V,E,\xi)\) where \(V\) is a finite non-empty set of \emph{vertices}, \(E\) is a finite set of \emph{edges} and
    \[
        \xi\colon E\to\{\{v,w\}\mid v,w\in V, \;v\neq w\}
    \] 
is an \emph{endpoint map}, which assigns each edge to an unordered pair of distinct vertices called the \emph{endpoints} of the edge. We often write \(v\in e\) as a shorthand for \(v\in \xi(e)\) to indicate that a vertex \(v\in V\) is an endpoint of an edge \(e\in E\). Likewise, we write \(e\cap f\) to denote \(\xi(e)\cap \xi(f)\), the set of endpoints shared by some edges \(e,f\in E\). The \emph{degree} of a vertex \(v\in V\) is the number of edges that have it as an endpoint:
\[\deg(v):=\lvert\{e\in E\mid v\in e\}\rvert.\]

\begin{convention}
    Whenever it is necessary to distinguish between the endpoints of an edge \(e\in E\), we may denote them by \(e^-\) and \(e^+\), so that
        \[
        \xi(e)=\{e^-,e^+\}.
        \]
    The assignment of the labels \(e^-\) and \(e^+\) is arbitrary, except for graphs arising from simple iterated graph systems, for which we will always use a canonical assignment specified in Definition \ref{Definition: Simple IGS}.
\end{convention}

    Let \(G=(V,E,\xi)\) be a graph. If \(V'\subset V\) and \(E'\subset E\) so that \(\xi(E')\subseteq E\), then the graph \(G'=(V',E',\xi\vert_{E'})\) is a \emph{subgraph} of of \(G\). Given a collection \(\{G_i=(V_i,E_i,\xi_i)\}_{i\in I}\) of subgraphs of \(G\), we define their union and intersection, respectively, as the subgraphs 
        \[
            \bigcup_{i\in I}G_{i}:=\big(\bigcup_{i\in I}V_i,\bigcup_{i\in I}E_i,\xi\vert_{\bigcup_{i\in I}E_i}\big)\quad\text{and}\quad\bigcap_{i\in I}G_{i}:=\big(\bigcap_{i\in I}V_i,\bigcap_{i\in I}E_i,\xi\vert_{\bigcap_{i\in I}E_i}\big).
        \]
        
    An alternating sequence \(\theta=[v_0,e_1,v_1,\ldots,e_k,v_k]\), where \(\{v_i\}_{i=0}^k\subset V\), \(\{e_i\}_{i=1}^k\subset E\) and \(k\in\N_0\), is a path from \(v_0\) to \(v_k\) if \(\xi(e_i)=\{v_{i-1},v_i\}\) for all \(1\leq i\leq k\). We say that an edge \(e\in E\) belongs to \(\theta\) and denote \(e\in\theta\) if \(e=e_i\) for some \(1\leq i\leq k\). Given any two non-empty sets of vertices \(A,B\subset V\), we denote the collection of paths joining \(A\) to \(B\) by
        \[
            \Theta(A,B):=\{\theta\mid \theta \text{ is a path from }v\in A\text{ to }w\in B\}.
        \]
    The length of a path is the number of edges it contains, and given a path \(\theta\), we denote its length by
        \[
            \len(\theta):=\lvert \{e\in E\mid e\in\theta\}\rvert.
        \]
    If there exists a path from \(v\) to \(w\) for every pair \(v,w\in V\), then we say that \(G\) is \emph{connected} and we define the \emph{path metric} \(d_G\) on \(V\) by 
        \[
            d_G(v,w):=\min_{\theta\in\Theta(\{v\},\{w\})}\len(\theta).
        \]
        
    We will often omit writing more than the first vertex of a path, simply denoting \(\theta=[v_0,e_1,e_2,\dots, e_k]\), as the rest of the vertices are unambiguously determined by the equalities \(\xi(e_i)=\{v_{i-1},v_i\}\).

    Given two paths \(\theta_1=[v_0,e_1,\ldots e_k,v_k]\) and \(\theta_2=[u_0,f_1,\ldots u_l,f_l]\) for which \(v_k=u_0\), we define the concatenation
        \[
            \theta_1*\theta_2:=[v_0,e_1,\ldots e_k,v_k,f_1,\ldots, f_l,u_l].
        \]

    A key tool in our arguments, is the concept of edge modulus of a family of paths. 
    
\begin{definition}\label{def: edgemodulus1}
    Let \(G=(V,E,\xi)\) be a graph. A function \(\rho\colon E\to [0,\infty)\) is called a density. The support of a density \(\rho\) is the set \(\supp(\rho):=\{e\in E\mid \rho(e)\neq 0\}\) and the \(\rho\)-length of any non-constant path \(\theta=[v_0,e_1,e_2,\dots,e_k]\) in \(G\) is
        \[
            L_\rho(\theta):=\sum_{i=1}^{k}\rho(e_i).
        \]
    Given a family of paths \(\Theta\) in \(G\), we say that a density \(\rho\) is \emph{\(\Theta\)-admissible} if \(L_\rho(\theta)\geq 1\) for every path \(\theta\in \Theta\). We denote the set of \(\Theta\)-admissible densities is by \(\Adm(\Theta)\). 
    
    For every \(p\in [1,\infty)\), we define the \(p\)-mass of a density \(\rho\) as the value
        \[
            \cM_p(\rho):=\sum_{e\in E}\rho(e)^p,
        \]
    and the \(p\)-edge modulus of a family of paths \(\Theta\) as the value
        \[
            \Mod_p(\Theta,G):=\inf_{\rho\in \Adm(\Theta)} \cM_p(\rho).
        \]
    A density \(\rho^*\in\Adm(\Theta)\) satisfying \(\cM_p(\rho^*)=\Mod_p(\Theta,G)\) is called \emph{\(p\)-optimal} for \(\Theta\).
\end{definition}

\begin{lemma}\label{Lemma: Existence of Minimizers}
    Let \(\Theta\) be a non-empty family of paths on a graph \(G=(V,E,\xi)\). Then, there exists a $p$-optimal $\rho^*\in\Adm(\Theta)$. If \(p>1\), then the $p$-optimal $\rho^*$ is unique. Further,  if \(\len(\theta)\geq 2\) for all \(\theta\in\Theta\), then \(\rho(e)<1\) for all \(e\in E\).
\end{lemma}

\begin{proof}
The existence follows from the optimization problem being convex, and since the optimum must be attained for $\rho$ that satisfy $\rho(e)\in [0,1]$. The uniqueness for $p>1$ follows from strict convexity.

Finally, assume that $\len(\theta)\geq 2$ for all edges $e$ and assume for the sake of contradiction that $\rho$ is optimal and admissible with $\rho(e)=1$ for some $e\in E$. Let $L(e)=\min_{e\in \theta\in \Theta} L_{\rho}(\theta)\geq 1$, where the minimum is attained since $\{L_\rho(\theta) : \theta\in \Theta\}$ is a discrete set. 

If $L(e)>1$, then we can decrease $\rho(e)$ slightly without altering admissibility. This contradicts minimality. If $L_\rho(e)=1$, then there is some edge $f\in E$ with $\rho(f)=0$. In this case define a new admissible weight function for $\epsilon \in (0,1/2)$
\[
\rho_\epsilon(f)=\begin{cases} 1-\epsilon & \text{ if  } f=e \\
\epsilon & \text{ if  } \rho(f)=0 \\
\rho(f) & \text{ otherwise }. \end{cases}
\]
It is direct to verify that $\rho_\epsilon$ is admissible. Further $\frac{d}{d\epsilon} \cM_p(\rho_\epsilon)|_{\epsilon = 0} = -p$.
\end{proof}

\begin{lemma}\label{Lemma: Continuity of modulus}
    Let \(\Theta\) be a non-empty family of paths on a graph \(G=(V,E,\xi)\) so that \(\len(\theta)\geq 2\) for all \(\theta\in\Theta\). Then, the mapping \(p\mapsto \Mod_p(\Theta,G)\) is continuous and strictly decreasing for \(p\in[1,\infty)\).
\end{lemma}

\begin{proof}
    Let \(\rho\) be the \(p\)-minimal density for \(\Theta\), for any \(p\in(1,\infty)\). By Lemma \ref{Lemma: Existence of Minimizers}, \(\rho\) is sub-unitary, and therefore the mapping \(t\mapsto\cM_t(\rho)\) is strictly decreasing for all \(t\in[1,\infty)\). Let \(1<q<p\), and let \(\nu\) and \(\rho\) be the \(q\)- and \(p\)-minimal densities for \(\Theta\), respectively. We then have that
        \[
        \Mod_p(\Theta,G)=\cM_p(\rho)<\cM_q(\rho)\leq \cM_q(\nu)=\Mod_q(\Theta,G),
        \]
    and therefore the mapping \(t\mapsto \Mod_t(\Theta,G)\) is strictly decreasing for all \(t\in [1,\infty)\). Consequently, to verify continuity, it is sufficient to show that
        \[
            \lim_{t\to p^-}\Mod_t(\Theta,G)=\Mod_{p}(\Theta,G)=\lim_{t\to p^+}\Mod_t(\Theta,G)
        \]
    for all \(p\in[1,\infty)\). For the first limit, let \(\rho\) be the \(p\)-optimal density for \(\Theta\). We then have that \(\cM_t(\rho)\geq \Mod_t(\Theta,G)\geq \Mod_p(\Theta,G)\) for all \(t\leq p\). Since the mapping \(t\mapsto \cM_t(\rho)\) is continuous, we have that \(\lim_{t\to p}\cM_t(\rho)=\cM_p(\rho)=\Mod_p(\Theta,G)\), and so it follows that
        \[
            \lim_{t\to p^-}\Mod_t(\Theta,G)=\Mod_p(\Theta,G).
        \]
    For the second limit let \(\{\rho_i\}_{i=1}^\infty\) be the sequence of \((p+1/i)\)-minimal densities for \(\Theta\). Since \(E\) is finite and \(\rho_i(e)\in[0,1)\) for all \(i\in\N\) and \(e\in E\), we may use a standard diagonal argument to pick a sub-sequence \(\{\rho_{i_j}\}_{j=1}^\infty\) for which \(\{\rho_{i_j}(e)\}_{j=1}^\infty\) converges for every \(e\in E\). Define the density \(\nu\colon E\to [0,1]\) by \(\nu(e):=\lim_{j\to\infty}\rho_{i_j}(e)\) for all \(e\in E\). We then have by construction that
        \[
            \lim_{t\to p^+}\Mod_t(\Theta,G)=\lim_{i\to\infty}\cM_{p+\frac{1}{i}}(\rho_i)=\cM_p(\nu).
        \]
    Since \(\Mod_t(\Theta,G)\leq \Mod_p(\Theta,G)\leq \cM_p(\nu)\) for all \(t\geq p\), it then follows that
        \[
            \lim_{t\to p^+}\Mod_t(\Theta,G)=\Mod_p(\Theta,G).
        \]
\end{proof}

For some of our arguments, we need to show that certain admissible densities are minimal. To this end, we use the dual relationship between moduli of path families and \emph{resistances} of flows. Let \(G=(V,E,\xi)\) be a graph and let \(A,B\subset V\) be non-empty and disjoint. The divergence of a function \(\cJ\colon V\times E\to \R\) at a vertex \(v\in V\) is defined as
        \[
            \divr(\cJ)(v):=\sum_{e\in E:v\in e}\cJ(v,e).
        \]
        
    A function \(\cJ\colon V\times E\to \R\) is called a \emph{flow} from \(A\) to \(B\) if it satisfies the following conditions:
    \begin{enumerate}
        \item \(\cJ(e^-,e)=-\cJ(e^+,e)\) for all \(e\in E\),
        \item \(\divr(\cJ)(v)=0\) for all \(v\not\in A\cup B\),
        \item \(\cJ(v,e)=0\) whenever \(v\not\in e\).
    \end{enumerate}
    For a flow $\cJ$ and $e\in E$ we write $|\cJ(e)|:=|\cJ(e^\pm,e)|$, where the choice of sign is irrelevant by (1). A flow \(\cJ\) from \(A\) to \(B\) is called a \emph{unit flow} if the net flux leaving \(A\) equals one, that is
        \[
            \sum_{v\in A}\divr(\cJ)(v)=1,\quad\text{or equivalently}\quad\sum_{v\in B}\divr(\cJ)(v)=-1.
        \]
    Let \(p\in(1,\infty)\) and denote it's dual exponent by \(q:=p/(p-1)\). The \emph{\(p\)-energy} of a unit flow \(\cJ\) is
        \[
            \cE_q(\cJ):=\sum_{e\in E}\lvert\cJ(e)\rvert^q,\quad\text{where }\quad\lvert\cJ(e)\rvert:=\lvert\cJ(e^-,e)\rvert=\lvert\cJ(e^+,e)\rvert.
        \]
    The \(p\)-resistance between \(A\) and \(B\) is defined as the value
        \[
            \Res_p(A,B,G):=\inf\{\cE_q(\cJ)\mid\cJ \text{ is a unit flow from }A\text{ to }B\}
        \]

The following proposition presents the dual relationship between modulus and resistance, see e.g. \cite[Theorem 5.1]{NakamuraYamasakiDuality}.
In the present work, however, we need a slightly more general version that also works for some disconnected graphs.
It follows fairly directly from the standard duality result by a slight modification of the graph involved.

\begin{proposition}\label{Proposition: Duality}
    Let $p\in (1,\infty)$ and $q = p/(p-1)$. Let \(G=(V,E,\xi)\) be a graph and let \(A,B\subset V\) be a pair of disjoint non-empty subsets such that $\Theta(A,B) \neq \varnothing$.
    Then
        \[
            \Mod_p(\Theta(A,B),G)^{\frac{1}{p}}\cdot\Res_p(A,B,G)^{\frac{1}{q}}=1.
        \]
    Further, if \(\cJ\) is a unit flow from \(A\) to \(B\) and \(\rho\in \Adm(\Theta(A,B))\), such that
        \[
            \mathcal{M}_p(\rho)^{\frac{1}{p}}\cdot\mathcal{E}_q(\mathcal{\cJ})^{\frac{1}{q}}=1,
        \]
    then both $\cJ$ and $\rho$ are optimal. For optimal $\cJ$ and $\rho$ we have
        \begin{equation}\label{flow_mod_vanish}
            \lvert \mathcal{F}(e) \rvert = \Mod_p(\Theta(A,B),G)^{-1}\rho(e)^{p-1}
        \end{equation}
\end{proposition}
\begin{proof} 
As stated, the connected case follows directly from \cite[Theorem 5.1]{NakamuraYamasakiDuality} and \cite[Lemma 4.13]{anttila2024constructions}. So, we are left to consider the case of a disconnected $G$. Let $\tilde{G}=(\tilde{V}, \tilde{E})$ be the graph obtained in these four steps: first identify all vertices in $A$ with a given vertex $a\in A$, then also identify all vertices in $B$ with a given vertex $b\in B$, remove possible loops and finally take the connected component containing $a$ and $b$ in the resulting graph.  It is not too hard to see that $\Mod_p(\Theta(A,B),G)=\Mod_p(\Theta(\{a\}, \{b\}, \tilde{G})$ and $\Res_p(\Theta(\{a\}, \{b\}, \tilde{G}))=\Res_p(A,B,G)$. 
\end{proof}

\section{Construction of Laakso-type fractal spaces}\label{Section: Construction of Laakso-type fractal spaces}

In this section, we present the construction of Laakso-type fractal spaces as limit spaces of iterated graph systems. Further discussion and detailed examples may be found in \cite{anttila2024constructions,ASEBShimizu}. In contrast to previous work, here we have adopted a notation that is explicitly suited to working with non-simple graphs.

\subsection{Iterated graph systems}\label{Subsection: Iterated graph systems} 

\begin{definition}
    An iterated graph system (IGS) consists of the following data:
    \begin{enumerate}
        \item A finite connected graph \(G_1=(V_1,E_1,\xi_1)\) called the \emph{generator}.
        \item A finite non-empty set \(I\) called the \emph{gluing set}.
        \item A collection of injective functions \(\{\phi_{v,e}\colon I\to V_1\}_{v\in e\in E_1}\) called \emph{gluing maps}, whose images, \(I_{v,e}:=\phi_{v,e}(I)\), satisfy the following whenever \(v\in e\in E_1\):
        \begin{enumerate}
            \item \(I_{e^-,e}\cap I_{e^+,e}=\varnothing\).
            \item There are no edges \(f\in E_1\) with \(\xi(f)\subset I_{v,e}\). 
        \end{enumerate}
    \end{enumerate}
    We refer to the pair \((I,\{\phi_{v,e}\colon I\to V_1\}_{v\in e\in E_1})\) as the \emph{gluing rule}. 
\end{definition}

    An IGS determines an infinite sequence of \emph{replacement graphs} \((G_n)_{n\in \N}\), and for each \(G_n\), a corresponding collection of gluing maps \(\{\phi_{v,e}\colon I\to V_1\}_{v\in e\in E_n}\), recursively, as follows: Consider an IGS with a generator \(G_1=(V_1,E_1,\xi_1)\) and gluing rule \((I,\{\phi_{v,e}\}_{v\in e\in E_1})\). The first graph in the sequence is the generator \(G_1\) itself. Suppose that \(G_n=(V_n,E_n,\xi_n)\) and \(\{\phi_{v,e}\}_{v\in e\in E_n}\) have been constructed for some \(n\in\N\). Then, we define the graph \(G_{n+1}=(V_{n+1},E_{n+1},\xi_{n+1})\) and the collection of gluing maps \(\{\phi_{v,e}\}_{v\in e\in E_{n+1}}\) as follows:
        \begin{enumerate}
            \item Let \(V_{n+1}:=(V_1\times E_n)/\sim\) with the identification \((\phi_{v,e}(i),e)\sim(\phi_{v,f}(i),f)\) for each \(i\in I\), whenever \(e\) and \(f\) share an endpoint \(v\).
            \item Let \(E_{n+1}:= E_1\times E_n\), and define \(\xi_{n+1}(e,f):=\{[(e^-,f)],[(e^+,f)]\}\) for every \((e,f)\in E_{n+1}\).
            \item Define the collection of gluing maps \(\{\phi_{v,e}\colon I\to V_1\}_{v\in e\in E_{n+1}}\), by setting \(\phi_{[(v,f)],(e,f)}:=\phi_{v,e}\) for all \((e,f)\in E_{n+1}\) and \(v\in e\).
        \end{enumerate}
    To clean up the notation, we drop the parentheses and simply write \([v,e]\) to denote any vertex \([(v,e)]\in V_n\). For any edge \(e\in E_n\), we denote it's level by \(\lvert e\rvert:=n\).

The replacement graphs of an IGS admit a natural projective structure defined as follows: For every \(n\geq 1\), we define the projection \(\pi_{n+1}\colon V_{n+1}\cup E_{n+1}\to V_n\cup E_n\) by setting
    \[
        \pi_{n+1}([v,e]):=
        \begin{cases}
            e&\text{if}\;v\not\in I_{e^-,e}\cup I_{e^+,e}\\
            e^-&\text{if}\;v\in I_{e^-,e}\\
            e^+&\text{if}\;v\in I_{e^+,e}
        \end{cases}
        \quad\text{and}\quad \pi_{n+1}(e,f):=f
    \]
for any vertex \([v,e]\in V_{n+1}\) and any edge \((e,f)\in E_{n+1}\), respectively. Furthermore, we define \(\pi_1:= \id\vert_{V_1\cup E_1}\), and for all \(n\geq m \geq 1\), we define
    \[
        \pi_{n,m}:=\begin{cases}
            \id\vert_{V_n\cup E_n}&\text{if } n=m,\\
            \pi_{m+1}\circ\pi_{m+2}\circ\ldots\circ\pi_n&\text{if } n>m.
        \end{cases}
    \]

\begin{definition}\label{Definition: Simple IGS}
An IGS is called \emph{simple} if there are exactly two gluing maps, \(\phi_-,\phi_+\colon I\to V_1\), so that \(\{\phi_{e^-,e},\phi_{e^+,e}\}=\{\phi_-,\phi_+\}\) for every edge \(e\in E_1\). For a simple IGS, we fix the endpoint labels \(e^-\) and \(e^+\) for each edge \(e\in E_n\) so that
    \[
        \phi_{e^-,e}=\phi_-\quad\text{and}\quad \phi_{e^+,e}=\phi_+.
    \]
We denote the gluing sets, higher-order gluing sets and the families of paths joining the \(n\)-th order gluing sets of a simple IGS by 
    \[
        I_\pm:=\phi_\pm(I),\quad I_\pm^{\smash{(n)}}:=\pi_{n,1}^{-1}(I_\pm)\quad\text{and}\quad\Theta^{\smash{(n)}}:=\Theta(I_-^{\smash{(n)}},I_+^{\smash{(n)}}),
    \]
respectively, and for each pair \((v,e)\) where \(e\in E_n\) and \(v\in \{e^-,e^+\}\), we define
    \[
        I_{v,e}^{(n)}:=\begin{cases}
            I_+^{(n)}&\text{if }v=e^+\\
            I_-^{(n)}&\text{if }v=e^-
        \end{cases}
    \]
For convenience, we also introduce the constant 
\[L_*:=\dist(I_-,I_+,d_{G_1}).\]

Usually, we will need to impose some minimal further assumptions on an IGS. We collect these in the following list. We say that a simple IGS is:
    \begin{enumerate}
        \item \emph{doubling}, if for every vertex \(v\in I_-\cup I_+\) we have \(\deg(v)=1\); 
        
        \item \emph{non-degenerate}, if there are no edges \(e\in E_1\) for which \(e^-\in I_\pm\) and \(e^+\in I_\mp\); 
        
        \item \emph{symmetric}, if there exists a graph isomorphism \(\eta:G_1\to G_1\)   for which \(\eta\circ\phi_\pm = \phi_\mp\);
    \end{enumerate}
\end{definition}

We review some properties of our construction.

\begin{lemma}[{\cite[Lemma 3.21]{anttila2024constructions}}]\label{Lemma: Doubling degree bound}
If a simple IGS is doubling then
    \[
    \sup_{n \in \N} \max_{v\in V_n}\deg(v)=\max_{v\in V_1}\deg(v).
    \]
\end{lemma}

The graphs $G_n$ are in a sense self-similar. This can be explicated using the following mappings. For every \(n,m\in \N\) and \(e\in E_n\), we define the similarity map \(\sigma_{e,m}\colon V_m\cup E_m\to V_{n+m}\cup E_{n+m}\) inductively, by first setting
        \[
            \sigma_{e,1}(v):=[v,e]\quad\text{and}\quad\sigma_{e,1}(f):=(f,e)
        \]
    for all \(v\in V_1\) and \(f\in E_1\), respectively, and then setting
        \[
            \sigma_{e,{m+1}}([v,f]):=[v,\sigma_{e,m}(f)]\quad\text{and}\quad \sigma_{e,m+1}((g,f)):=(g,\sigma_{e,m}(f))
        \]
    for all \([v,f]\in V_{m+1}\) and \((g,f)\in E_{m+1}\), respectively. Furthermore, we denote
    \[
        e\cdot V_m:=
            \begin{cases}
                \{e^-,e^+\}&\text{if }m=0\\ 
                \sigma_{e,m}(V_m)&\text{if }m\geq 1
            \end{cases}
        \quad\text{and}\quad 
        e\cdot E_m:=
            \begin{cases}
                e&\text{if }m=0\\
                \sigma_{e,m}(E_m)&\text{if }m\geq 1
            \end{cases}
    \]
    and define \(e\cdot G_m\) as the subgraph of \(G_{n+m}\) with vertices \(e\cdot V_m\) and edges \(e\cdot E_m\).

The following Proposition is equivalent to \cite[Proposition 3.11]{anttila2024constructions}, with a minor difference in \ref{SM2}, since here we also allow graphs that are non-simple. As the proof does not change, we omit it.

\begin{proposition}\label{prop: similarity maps}
    The similarity mapping \(\sigma_{e,m}\) has the following properties for every \(n,m\in \N\) and each \(e\in E_n\):

    \begin{enumerate}[label={\color{blue}{\textup{$($SM\arabic*}$)$}}, widest=a, leftmargin=*]
        \item\label{SM1} The mapping is injective and the collections \(\{f\cdot V_m\}_{f\in E_n}\) and \(\{f\cdot E_m\}_{f\in E_n}\) cover \(V_{n+m}\) and \(E_{n+m}\), respectively. Moreover, if \(v_-,v_+\in V_m\) are any two vertices, then \(\{v_-,v_+\}=\xi_m(f)\) for some \(f\in E_m\) if and only if \(\{\sigma_{e,m}(v_-),\sigma_{e,m}(v_+)\}=\xi_{n+m}(\sigma_{e,m}(f))\) and 
            \[
            \phi_{v_\pm,f}=\phi_{\sigma_{e,m}(v_\pm),\sigma_{e,m}(f)}
            \]
        \item\label{SM2} For distinct edges \(e,f\in E_n\) the subsets \(e\cdot V_m\) and \(f\cdot V_m\) intersect if and only if \(e\) and \(f\) share an endpoint \(v\). Moreover, the intersection is
            \[
            \bigcup_{v\in e\cap f}\sigma_{e,m}\big(I_{v,e}^{(m)}\big)=\pi_{n+m,n}^{-1}(e\cap f)=\bigcup_{v\in e\cap f} \sigma_{f,m}(I_{v,f}^{(m)}).
            \]
        \item\label{SM3} For every \(e\in E_n\) we have that \(e\cdot E_m=\pi_{n+m,n}^{-1}(e)\cap E_{n+m}\). In particular, \(\{f\cdot E_m\}_{f\cdot E_n}\) is a partition of \(E_{n+m}\).
    \end{enumerate}
\end{proposition}

For any two edges \(e\in E_n\) and \(f\in E_m\) we define the operation
        \[
            e\cdot f:=\sigma_{e,m}(f).
        \]
The operation is associative, and gives us a convenient way of representing any edge \(e\in E_n\) uniquely as a product \(e=e_1\cdot e_2\cdot \ldots \cdot e_n\), where \(e_i\in E_1\) for each \(1\leq i\leq n\). Given a path \(\theta=[v_1,f_1,\ldots, f_k,v_k]\) on \(G_1\), we also define the path
    \[
        e\cdot \theta:=[\sigma_{e,1}(v_0),\sigma_{e,1}(f_1),\ldots,\sigma_{e,1}(f_k),\sigma_{e,1}(v_k)].
    \]

The following proposition is roughly equivalent to \cite[Proposition 3.16]{anttila2024constructions}, with some differences in the definition of a path. We omit the proof, since it is very similar, but mention that the indices $k_i,s_i$ correspond to types of stopping times, and the intervals of indices $[k_i,s_i]$ to crossing times of the subgraphs $f_i\cdot G_m$. 

\begin{proposition}\label{Proposition: Path lift}
    Let \(\theta=[v_0,e_1,v_1,\dots ,e_k,v_k]\) be a path in \(G_{n+m}\) and \(A,B\subset V_n\) be non-empty disjoint sets. If \(v_0\in \pi_{n+m,n}^{-1}(A)\) and \(v_k\in \pi_{n+m,n}^{-1}(B)\), then there are numbers \(0\leq k_1 < s_1 \leq k_2 < s_2\leq,\dots,\leq k_l < s_l\leq k\), vertices \(u_0,\dots ,u_l\in V_n\) and edges \(f_1,\dots ,f_l\in E_n\) so that the following conditions hold
    \begin{enumerate}
        \item \([u_0,f_1,u_1\dots,f_l,u_l]\in \Theta(A,B)\),
        \item For $i=1,\dots,l$, the sub-path \(\theta_i=[v_{k_i},e_{k_i+1},\dots ,e_{s_i}, v_{s_i}]\) of \(\theta\) is a path from  \(v_{k_i}\in\pi_{n+m,n}^{-1}(u_i)\) to \(v_{s_i}\in\pi_{n+m,n}^{-1}(u_{i+1})\) and is contained in \(f_i\cdot G_m\).
    \end{enumerate}
\end{proposition}

\begin{lemma}\label{Lemma: Non-degenerate path length}
    Let \(\theta=[v_0,e_1,\ldots,e_k,v_k]\in\Theta^{(n)}\) be a simple path. If the IGS is non-degenerate, then \(\len(\theta)\geq L_*^n\ge 2^n\). 
\end{lemma}
We omit the proof since this is quite direct; see \cite[Lemma 3.18]{anttila2024constructions} for a detailed argument. 

\begin{lemma}\label{Lemma: Interior edge}
If a simple IGS is non-degenerate and doubling, then for every \(e\in E_n\) there exists \(f_1,f_2\in e\cdot E_2\) such that \(f_1\cap f_2=\varnothing\) and \(f_1\cap g=\varnothing\) for all \(g\in E_{n+2}\setminus (e\cdot E_2)\).
\end{lemma}
\begin{proof}
    Choose a simple path \(\theta=[v_0,e_1,v_1,\ldots,e_k,v_k]\in\Theta^{(2)}\). Since the IGS is non-degenerate, it follows from Lemma \ref{Lemma: Non-degenerate path length} that \(\len(\theta)=k\geq 4\). Let \(f_1=e\cdot e_2\) and \(f_2=e\cdot e_4\). As \(\theta\) is simple, we have that \(e_2\cap e_4=\varnothing\) and consequently \(f_1\cap f_2=\varnothing\). By the doubling property, if \(v_i\in I^{\smash{(2)}}_\pm\), then \(\degr(v_i)=1\), and hence \(e_i=e_{i+1}\). Thus, as \(\theta\) is simple, it follows that \(v_i\not\in I_\pm^{\smash{(2)}}\) for all \(1\leq i\leq k-1\). Particularly, we have that \(e_2^-,e_2^+\not\in I^{\smash{(2)}}_\pm\) and hence it follows from \ref{SM2} that \(f_1\cap g =\varnothing\) for all \(g\in E^{n+2}\setminus(e\cdot E^2)\).
\end{proof}

\subsection{Limit space} The \emph{symbol space} of an IGS is the collection of projective sequences
    \[
        \Sigma := \big\{(e_i)_{i=1}^\infty \mid e_i\in E_i\text{ and }\pi_{i+1}(e_{i+1})=e_i\text{ for all }i\in\N\big\}.
    \]
For each \(n\in \N\) and \(e\in E_n\), we denote \(\Sigma_e:=\{(e_i)_{i=1}^\infty\in \Sigma\mid e_n=e\}\). The symbol space admits a natural \emph{word metric}, given as
    \[
        \delta_{\Sigma}((e_i)_{i = 1}^\infty,(f_i)_{i = 1}^\infty) :=
        \begin{cases}
            0 & \text{ if } e_i = f_i \text{ for all } i \in \N,\\
            2^{-\max \left\{0,\, k \, \big| \,  (e_i)_{i = 1}^k = (f_i)_{i = 1}^k \right\}} & \text{ otherwise.}
        \end{cases}
    \]
By construction, \((\Sigma, \delta_\Sigma)\) is a compact metric space. The \emph{limit space} of an IGS is defined as the quotient space \(X:=\Sigma/\sim\), where we identify
    \[
        (e_i)_{i=1}^\infty\sim(f_i)_{i=1}^\infty\iff e_i\cap f_i\neq \varnothing\text{ for all }i\in\N.
    \]
The canonical projection \(\chi\colon \Sigma \to X\) is given by \(\chi((e_i)_{i=1}^\infty):=[(e)_{i=1}^\infty]\), and we denote \(X_e:=\chi(\Sigma_e)\) for all \(n\in\N\) and \(e\in E_n\). For a simple IGS, we also define the sets \(X_-,X_+\subset X\) by setting
    \[
        X_\pm=\{[(e_i)_{i=1}^\infty]\mid e_n\cap I_{\pm}^{(n)}\neq\varnothing\;\;\text{for all}\;\;n\in\N\}
    \]

\begin{remark}\label{Remark: Product Decomposition}
By construction, every edge \(e\in E_n\) may be uniquely represented as a product \(e_1\cdot e_2\cdot \ldots \cdot e_n =e\), where \(e_i\in E_1\) for all \(1\leq i \leq n\). Likewise, every such product corresponds to a unique edge \(e_1\cdot e_2\cdot \ldots \cdot e_n\in E_n\). Consequently, for each \(n\in \N\), we may identify the edge set \(E_n\) with the set of sequences
    \[
        E_1^n:=\{(e_i)_{i=1}^n\mid e_i\in E_1\text{ for all }1\leq i \leq n\}
    \]
via the mapping \((e_i)_{i=1}^n\mapsto e_1\cdot e_2\cdot \ldots \cdot e_n\). Similarly, we may identify the symbol space \(\Sigma\) with the set of sequences 
    \[
        E_1^\N=\{(e_i)_{i=1}^\infty\mid e_i\in E_1\text{ for all } i\in \N\}
    \]
via the mapping \((e_i)_{i=1}^\infty \mapsto (e_1\cdot e_2\cdot \ldots \cdot e_i)_{i=1}^\infty\). It is often very convenient to think of the edges in terms of the corresponding product decompositions, which motivates the following definition.
\end{remark}

\begin{definition}\label{def:Tdef}
    We define mappings \(T:E_1^n \to E_n\) and \(T:E_1^\N \to \Sigma\) by
        \[
            T((e_i)_{i=1}^n):=\begin{cases}
                \;e_1\cdot e_2\cdot \ldots \cdot e_n&\text{if }n\in \N,\\
                (e_1\cdot e_2\cdot \ldots \cdot e_i)_{i=1}^\infty&\text{if }n=\infty.
            \end{cases}
        \]
    Furthermore, if \(x=T((e_i)_{i=1}^n)\), we denote \(T^{-1}(x)_i=e_i\in E_1\).
\end{definition}

As is made clear by the discussion in Remark \ref{Remark: Product Decomposition}, the inverse map \(T^{-1}\) is well defined. Moreover, we may conveniently refer to a specific term \(T^{-1}(x)_i\) in the product decomposition of an element \(x\in \Sigma\cup(\bigcup_{n=1}^\infty E_n)\) and we have
    \[
        x=\begin{cases}
            \,\ T^{-1}(x)_1\cdot \ldots \cdot T^{-1}(x)_n&\text{if }x\in E_n\text{ for some }n\in \N,\\
            \big(T^{-1}(x)_1\cdot \ldots \cdot T^{-1}(x)_i\big)_{i=1}^\infty&\text{if }x\in \Sigma.
        \end{cases}
    \]

\begin{lemma}\label{Lemma: Intersecting Sequences}
    Let \(\Sigma\) be the symbol space of a doubling and symmetric IGS. Suppose that \((e_i)_{i=1}^\infty
    \sim (f_i)_{i=1}^\infty\in \Sigma\) are such that \(e_{n}\cap f_{n}\neq \varnothing\) for some \(n\in\N\). Then there exists \(m\leq n\), such that 
        \[
            T^{-1}((e_i)_{i=1}^\infty)_i=
            \begin{cases}
                T^{-1}((f_i)_{i=1}^\infty)_i&\text{for } 1\leq i< m\\
                T^{-1}((f_i)_{i=1}^\infty)_i \;\text{ or }\; \eta\circ T^{-1}((f_i)_{i=1}^\infty)_i&\text{for }m< i\leq n
            \end{cases}
        \]
\begin{proof}
    Let \(m=\min\{n,k\mid e_k\neq f_k\}\) and assume that \(n>1\), as otherwise there is nothing to prove. As \(e_i=f_i\) for all \(1\leq i<m\), we must have that 
        \[
            T^{-1}((e_i)_{i=1}^\infty)_i=T^{-1}(e_i)_i=T^{-1}(f_i)_i=T^{-1}((f_i)_{i=1}^\infty)_i
        \]
    for all \(1\leq i <m\). Since \(e_n\cap f_n\neq\varnothing\), we have by \ref{SM2} that for all \(m\leq i< n\) there exists a vertex \(v_i\in e_i\cap f_i\) and an element \(a_i\in I\), so that 
        \[
            e_{i+1}\cap f_{i+1}=[\phi_{v_i,e_i}(a_i),e_i]=[\phi_{v_i,f_i}(a_i),f_i].
        \] 
    Since the IGS is symmetric, we have that \(\phi_{v_i,e_i}(a_i)\in\{\phi_{v_i,f_i}(a_i),\:\eta\circ \phi_{v_i,f_i}(a_i)\}\). By definition, the vertices \(\phi_{v_i,e_i}(a_i)\), \(\phi_{v_i,f_i}(a_i)\) and \(\eta\circ \phi_{v_i,f_i}(a_i)\) are endpoints of the edges \(\sigma_{e_i,1}^{-1}(e_{i+1})\), \(\sigma_{f_i,1}^{-1}(f_{i+1})\) and \(\eta\circ \sigma_{f_i,1}^{-1}(f_{i+1})\), respectively. Moreover, we have that \(\sigma_{e_i,1}^{-1}(e_{i+1})=T^{-1}((e_i)_{i=1}^\infty)_{i+1}\) and \(\sigma_{f_i,1}^{-1}(f_{i+1})=T^{-1}((f_i)_{i=1}^\infty)_{i+1}\). Therefore, the edge \(T^{-1}((e_i)_{i=1}^\infty)_{i+1}\) has a common endpoint \(\phi_{v_i,e_i}(a_i)\) with either \(T^{-1}((f_i)_{i=1}^\infty)_{i+1}\) or \(\eta\circ T^{-1}((f_i)_{i=1}^\infty)_{i+1}\). However, since the IGS is doubling, we have that \(\deg(\phi_{v_i,e_i}(a_i))=1\), and so there exists a unique edge \(\mathfrak{e}(\phi_{v_i,e_i}(a_i))\in E_1\) with \(\phi_{v_i,e_i}(a_i)\) as one of it's endpoints. Therefore, for all \(m<i\leq n\), we have that
        \[
            T^{-1}((e_i)_{i=1}^\infty)_{i}=\mathfrak{e}(\phi_{v_{i-1},e_{i-1}}(a_{i-1}))\in\{T^{-1}((f_i)_{i=1}^\infty)_{i},\;\eta\circ T^{-1}((f_i)_{i=1}^\infty)_{i}\}
        \]
\end{proof}

\end{lemma}

\begin{remark}\label{rmk:intersection}
    Let $X_{e,\pm}=\{[(e_i)_{i=1}^\infty
] \in X: e_n = e, e_k \in e\cdot I^{k-n}_{\pm}, k>n\}$. The proof of Lemma \ref{Lemma: Intersecting Sequences} shows  that for two adjacent edges $e,f\in E_n$  we have $X_e \cap X_f \subset X_{f,\pm}$.
\end{remark}

\begin{definition}
    A symmetric Laakso-type fractal space is the limit space \(X\) of a simple IGS that is doubling, non-degenerate and symmetric.
\end{definition}

\section{Construction of metrics}\label{Section: Construction of Metrics}
In this section, we construct a family of metrics on the limit space. Most importantly, this construction yields a subfamily of quasisymmetrically equivalent metrics with visual-metric-like properties, and as we will show in Section \ref{Section: Moduli and Conformal Dimension}, this subfamily constitutes a dimensionally exhaustive subset of the related conformal gauge.

In contrast to an earlier construction \cite{anttila2024constructions}, the present one allows us to remove the uniform scaling assumption, and to handle metrics that are not necessarily quasiconvex. Similar constructions that involve the chaining of balls, sets, or pairs of points have appeared earlier in many places, such as  \cite{bonkmeyer, BonkSchramm, Carrasco, SemmesBilip}. 

For the rest of this section, we fix a doubling, non-degenerate IGS with generator \(G_1=(V_1,E_1,\xi_1)\). Later on, we also require the IGS to be symmetric, in which case we denote the associated symmetry map by \(\eta\).

\subsection{Density cascade}
A density \(\rho\colon E_1\to (0,1)\) induces a sequence \((\rho_n)_{n\in\N}\) of densities as a \emph{multiplicative cascade}, where \(\rho_n\colon E_n\to (0,1)\) is defined recursively for each \(e\in E_n\) as
    \begin{equation}\label{Definition: Density Cascade}
        \rho_n(e):=
        \begin{cases}
            \rho(e)&\text{if }\;n=1, \\
            \rho_{n-1}(\pi_n(e))\cdot \rho_1(\sigma_{\pi_n(e),1}^{-1}(e)) &\text{if }\;n>1.
        \end{cases}
    \end{equation}
Equivalently, each density in the sequence can be given as a product
    \begin{equation}\label{Definition: Density Product}
        \rho_n(e)=\prod_{i=1}^n\rho(T^{-1}(e)_i).
    \end{equation}
From now on, given any density \(\rho\colon E_1\to (0,1)\), we refer by \(\rho_n\) to the density \(\rho_n\colon E_n\to (0,1)\) given equivalently by equations \eqref{Definition: Density Cascade} and \eqref{Definition: Density Product}. For convenience, we also introduce the constants:
    \[
        M_\rho:=\max_{e\in E_1}\rho_1(e)\;\;\text{and}\;\;m_\rho:=\min_{e\in E_1}\rho_1(e).
    \]
Note, that whenever \(e\in E_n\) and \(f\in e\cdot E_k\) for some \(k\in\N\), it follows directly from the construction that
    \begin{equation}\label{Equation: Max-min level bound}
        m_\rho^k\rho_n(e)\leq\rho_{n+k}(f)\leq M_\rho^k\rho_n(e).
    \end{equation}

Recall that for $e\in E_n$, we write $\lvert e \rvert:=n$. Whenever \(n>m\) and \(\Omega\in\bigcup_{i=1}^m E_i\), we define the sub-graph covered by \(\Omega\) in \(G_m\) as 
    \[
        p_n(\Omega):=\bigcup_{e\in \Omega}e\cdot G_{n-\lvert e \rvert}.
    \]
For every pair of vertices \(v,w\in V_n\) we define the collection
    \[
        \Lambda_{v,w}^n:=\left\{ \Omega\in\bigcup_{i=1}^nE_i\ \middle\vert\ v \text{ and } w\text{ are connected by a path in \(p_n(\Omega)\) } \right\}.
    \]
We then define the metric \(d_{\rho,n}\colon V_n\times V_n\to[0,\infty)\) as
\[
    d_{\rho,n}(v,w):=
    \begin{cases}
        0 &\text{if }\; v=w, \\
        \displaystyle\inf_{\Omega\in \Lambda_{v,w}^n}\displaystyle\sum_{e\in\Omega}\rho_{\lvert e \rvert}(e)& \text{if }\;v\neq w.
    \end{cases}      
\]
For any pair \(e,f\in E_n\), we denote \(d_{\rho,n}(e,f):=\min\{d_{\rho,n}(v,w)\mid v\in e,\:w\in f\}\).

\begin{remark}
    To see that \(d_{\rho,n}\) satisfies the triangle inequality, observe for any three vertices \(v,w,u\in V_n\), if \(\Omega_1\in\Lambda_{v,w}^n\) and \(\Omega_2\in \Lambda_{w,u}^n\), then \(\Omega_1\cup\Omega_2\in\Lambda_{v,u}^n\). Since \[\sum_{e\in\Omega_1\cup\Omega_2}\rho_{\lvert e \rvert}(e)\leq\sum_{e\in\Omega_1}\rho_{\lvert e \rvert}(e)+\sum_{e\in\Omega_2}\rho_{\lvert e \rvert}(e),\]
    it follows by passing to the infimum, that that \(d_{\rho,n}(v,u)\leq d_{\rho_n}(v,w)+d_{\rho,n}(w,u).\)
\end{remark}

If $m>n$, we say that $x_m\in V_m$ is an ancestor of $x_n \in V_n$ if $\pi_{m,n}(x_m)=x_n$.

\begin{lemma}\label{Lemma: Distance Lemma 1}
    Let \(\rho\colon E_1\to (0,1)\) be \(\Theta^{(1)}\)-admissible and \(n,m\in \N\) so that \(n<m\).
    \begin{enumerate}[label={\color{blue}{\textup{$($DL\arabic*}$)$}}, widest=a, leftmargin=*]
        \item \label{Lemma: DL(1)} If \(x_m,y_m\in V_m\) are ancestors of \(x_n,y_n\in V_n\), and \(x_n\neq y_n\), then 
            \[
                d_{\rho,m}(x_m,y_m)=d_{\rho,n}(x_n,y_n).
            \]
        \item \label{Lemma: DL(2)} If \(e\in E_n\) and \(x_m,y_m\in e\cdot V_{m-n}\), then
            \[
                d_{\rho,m}(x_m,y_m)\leq \rho_n(e).
            \]
        \item \label{Lemma: DL(3)} If \(e,f\in E_n\), \(v\in e\) and \(w\in f\), then
            \[
                d_{\rho,n}(v,w)-\rho_n(e)-\rho_n(f)\leq d_{\rho,n}(e,f)\leq d_{\rho,n}(v,w).
            \]
        \item \label{Lemma: DL(4)} If \(e,f\in E_m\), then
            \[
            \begin{split}
                &d_{\rho,n}(\pi_{m,n}(e),\pi_{m,n}(f))\leq d_{\rho,m}(e,f)\\
                &\leq d_{\rho,n}(\pi_{m,n}(e),\pi_{m,n}(f))+\rho_n(\pi_{m,n}(e))+\rho_n(\pi_{m,n}(f)).
            \end{split}
            \]
    \end{enumerate}
\end{lemma}
\begin{proof}
    We begin with \ref{Lemma: DL(1)}. It is enough to only consider the case \(m=n+1\) since the general case then follows by induction. By construction, we have that \(\Lambda^n_{x_n,y_n}\subset\Lambda^{n+1}_{x_{n+1},y_{n+1}}\), and therefore it follows directly that
        \[
            d_{\rho,n+1}(x_{n+1},y_{n+1})\leq d_{\rho,n}(x_n,y_n).
        \]
    For the reversed inequality, let \(\Omega\in\Lambda^{n+1}_{x_{n+1},y_{n+1}}\). By definition, there is a simple path \(\theta\) in \(p_{n+1}(\Omega)\) joining \(x_{n+1}\) to \(y_{n+1}\). By Proposition \ref{Proposition: Path lift}, since \(x_{n+1}\neq y_{n+1}\), the path \(\theta\) admits a lift to \(G_n\); that is, there exists a simple path \(\vartheta=[x_n,e_1,\ldots,e_k,y_n]\) in \(G_n\) and, for each \(1\leq i\leq k\), a simple path \(\theta_i\in\Theta^{(1)}\) such that \(\{e_i\cdot\theta_i\}_{i+1}^k\) is a collection of edge-disjoint sub-paths of \(\theta\). As \(\rho\) is \(\Theta^{(1)}\)-admissible, we have for every \(1\leq i\leq k\) that
        \begin{equation}\label{Equation: DL1 Inequality 1}
            \rho_n(e_i)\leq \rho_n(e_i)\cdot L_{\rho_1}(\theta_i)=L_{\rho_{n+1}}(e_i\cdot\theta_i)=\sum_{e\in e_i\cdot\theta_i}\rho_{n+1}(e).
        \end{equation}
    Define
        \[
            \Omega_u:=\{e_i\in\vartheta\mid e_i\not\in p_n(\Omega \setminus E_{n+1})\},
        \]
    so that \(\vartheta\subset p_n((\Omega \setminus E_{n+1}) \cup\Omega_u)\) and hence \((\Omega \setminus E_{n+1}) \cup\Omega_u \in\Lambda^n_{x_n,y_n}\). For each \(e_i\in\Omega_u\), the corresponding sub-path \(e_i\cdot\theta_i\) is contained in \(p_{n+1}(\Omega)\) but cannot meet \(p_{n+1}(\Omega \setminus E_{n+1})\) as otherwise there would be some \(f\in\Omega \setminus E_{n+1}\) with an edge of \(e_i\cdot\theta_i\) in \(f\cdot E_{n+1-\lvert f\rvert}\) which would force \(e_i\in f\cdot E_{n-\lvert f\rvert}\subset  p_n(\Omega \setminus E_{n+1})\), contradicting \(e_i\in\Omega_u\). Consequently, if \(e_i\in\Omega_u\), then every edge of \(e_i\cdot\theta_i\) must belong to \(p_{n+1}(\Omega\cap E_{n+1})\) and therefore to \(\Omega \cap E_{n+1}\). 
    
    By equation \eqref{Equation: DL1 Inequality 1}, we have that
        \[
            \sum_{e\in\Omega_u}\rho_n(e)\leq \sum_{e\in \Omega \cap E_{n+1}}\rho_{n+1}(e),
        \]
    and therefore
        \[
            \sum_{e\in(\Omega \setminus E_{n+1}) \cup\Omega_u}\rho_{\lvert e\rvert}(e)\leq \sum_{e\in\Omega \setminus E_{n+1}}\rho_{\lvert e\rvert}+\sum_{e\in\Omega\cap E_{n+1}}\rho_{\lvert e\rvert} =\sum_{e\in\Omega}\rho_{\lvert e\rvert},
        \]
    and finally, since \((\Omega \setminus E_{n+1})\cup\Omega_u\in\Lambda^n_{x_n,y_n}\), we obtain
        \[
            d_{\rho,n}(x,y)\leq\sum_{e\in \Omega \setminus E_{n+1}\cup\Omega_u}\rho_{\lvert e\rvert }(e)\leq \sum_{e\in\Omega}\rho_{\lvert e\rvert}.
        \]
    Since \(\Omega\in\Lambda^{n+1}_{x_{n+1},y_{n+1}}\) was arbitrary, it follows that
        \[
            d_{\rho,n}(x_n,y_n)\leq d_{\rho,n+1}(x_{n+1},y_{n+1}).
        \]

    We move on to \ref{Lemma: DL(2)}. For any \(e\in E_n\), the sub-graph \(p_{m}(\{e\})=e\cdot G_{m-n}\) is connected and contains \(e\cdot V_{m-n}\). Thus, if \(x_m,y_m\in V_{m-n}\), then \(\{e\}\in\Lambda^m_{x_m,y_m}\) and consequently
        \(
            d_{\rho,m}(x_m,y_m)\leq \rho_n(e).
        \)
    \vspace{5pt}
    
    Next, we prove \ref{Lemma: DL(3)}. Directly from the definition of \(d_{\rho,n}\) for edges, we obtain
        \[
            d_{\rho,n}(e,f)=\min_{x\in e,y\in w}d_{\rho,n}(x,y)\leq d_{\rho,n}(v,w).
        \]
    On the other hand, since \(\{e\}\in\Lambda^n_{e^-,e^+}\) and \(\{f\}\in\Lambda^n_{f^-,f^+}\) it follows from the triangle inequality that
        \[
        \begin{split}
         d_{\rho,n}(v,w) &\leq d_{\rho,n}(e,f) + d_{\rho,n}(e^-,e^+) + d_{\rho,n}(f^-,f^+)\\
         &\leq d_{\rho,n}(e,f) + \rho_n(e) + \rho_n(f).
         \end{split}
        \]
    Thus, we obtain the remaining inequality
        \[
            d_{\rho,n}(v,w)-\rho_n(e)-\rho_n(f)\leq d_{\rho,n}(e,f).
        \]

    Finally, we prove \ref{Lemma: DL(4)}. We begin with the first inequality. First, note that if \(\pi_{m,n}(e)\cap\pi_{m,n}(f)\neq\varnothing\), then it follows immediately that 
        \[
            d_{\rho,n}(\pi_{m,n}(e),\pi_{m,n}(f))=0\leq d_{\rho,m}(e,f).
        \]
    Therefore, suppose that \(\pi_{m,n}(e)\cap \pi_{m,n}(f)=\varnothing\) and let \(x,y\in V_m\) be such that \(d_{\rho,m}(e,f)= d_{\rho,m}(x,y)\). Let \(\Omega\in \Lambda_{x,y}^m\). By definition, there exists a path \(\theta\) joining \(x\) and \(y\) in \(p_m(\Omega)\), and by Proposition \ref{Proposition: Path lift} there exists ancestors \(x_m,y_m\) of some vertices \(x_n\in \pi_{m,n}(e)\) and \(y_n\in \pi_{m,n}(f)\), respectively, so that \(\theta\) has a sub-path from \(x\) to \(x_m\) and another sub-path from \(y\) to \(y_m\). Therefore we have that \(\Omega\in \Lambda_{x_m,y_m}^m\), and since \ref{Lemma: DL(1)} gives us that \(d_{\rho,m}(x_m,y_m)=d_{\rho,n}(x_n,y_n)\), it follows that
        \begin{equation}
            d_{\rho,n}(\pi_{m,n}(e),\pi_{m,n}(f))\leq d_{\rho,n}(x_n,y_n)= d_{\rho,m}(x_m,y_m)\leq \sum_{g\in \Omega}\rho_{\lvert g \rvert}(g).
        \end{equation}
    Since the collection \(\Omega\in \Lambda_{x,y}^m\) was arbitrary, we obtain the first inequality
        \[
            d_{\rho,n}(\pi_{m,n}(e),\pi_{m,n}(f))\leq d_{\rho,m}(e,f).
        \]
    For the second inequality, let \(x,y\in V_m\) be such that \(d_{\rho,m}(e,f)= d_{\rho,m}(x,y)\) and let \(v,w\in V_n\) be such that \(d_{\rho,n}(v,w)=d_{\rho,n}(\pi_{m,n}(e),\pi_{m,n}(f))\). Let \(\Omega \in \Lambda_{v,w}^n\) and define 
        \[
            \Omega_0:=\Omega\cup\{\pi_{m,n}(e),\pi_{m,n}(f)\}.
        \]
    Since \(e\in \pi_{m,n}(e)\cdot E_{m-n}\) and \(f\in \pi_{m,n}(f)\cdot E_{m-n}\), we have that \(\Omega_0\in\Lambda_{x,y}^m\) and therefore
        \begin{equation}
            d_{\rho,m}(e,f)\leq d_{\rho,m}(x,y)\leq \sum_{g\in\Omega_0}\rho_{\lvert g\rvert}(g)=\rho_n(\pi_{m,n}(e))+\rho_n(\pi_{m,n}(f))+\sum_{g\in \Omega}\rho_{\lvert g \rvert}(g).
        \end{equation}
     Since \(\Omega\in \Lambda_{v,w}^n\) was arbitrary, we obtain the second inequality
        \[
            d_{\rho,m}(e,f)\leq d_{\rho,n}(\pi_{m,n}(e),\pi_{m,n}(f))+\rho_n(\pi_{m,n}(e))+\rho_n(\pi_{m,n}(f)).
        \]
\end{proof}

\subsection{Limit metrics}
    Let \(\rho\colon E_1\to (0,1)\) be a \(\Theta^{(1)}\)-admissible density. We define the metric \(d_\rho\) on the limit space \(X\) of the IGS  by
        \[
            d_\rho([(e_i)_{i=1}^\infty],[(f_i)_{i=1}^\infty]):=\lim_{i\to \infty}d_{\rho,i}(e_i,f_i).
        \]
    The metric is well defined since the limit does not depend on the chosen representatives for \([(e_i)_{i=1}^\infty]\) and \([(f_i)_{i=1}^\infty]\), as will be shown in the next lemma.
\begin{lemma}\label{Lemma: Well-definedness of Metric}
    The function \(d_\rho\) is a well-defined metric on \(X\).
\end{lemma}
\begin{proof}
    Given \((e_i)_{i=1}^\infty,(f_i)_{i=1}^\infty,(h_i)_{i=1}^\infty\in \Sigma\) so that \((e_i)_{i=1}^\infty\sim (f_i)_{i=1}^\infty\), we have by definition that \(e_i\cap f_i\neq \varnothing\) for all \(i\in\N\). Since \(\{e_i\}\in \Lambda_{e_i^-,e_i^+}^i\) and \(\{f_i\}\in \Lambda_{f_i^-,f_i^+}^i\) for all \(i\in \N\), the triangle inequality of the metrics \(d_{\rho,i}\) gives us that
        \[
            \lvert d_{\rho,i}(e_i,h_i)-d_{\rho,i}(f_i,h_i)\rvert \leq d_{\rho,i}(e_i^-,e_i^+)+d_{\rho,i}(f_i^-,f_i^+)\leq\rho_i(e_i)+\rho_i(f_i),
        \]
    and since \(\rho_i(e_i),\rho_i(f_i)\stackrel{i\to \infty}{\longrightarrow}0\), we have that 
        \[
            \lim_{i\to \infty}d_{\rho,i}(e_i,h_i)=\lim_{i\to\infty}d_{\rho,i}(f_i,h_i).
        \]
    Therefore, if \((e_i)_{i=1}^\infty,(f_i)_{i=1}^\infty,(g_i)_{i=1}^\infty,(h_i)_{i=1}^\infty\in\Sigma\) are such that \((e_i)_{i=1}^\infty\sim(f_i)_{i=1}^\infty\) and \((h_i)_{i=1}^\infty\sim(g_i)_{i=1}^\infty\), then
        \[
            \lim_{i\to \infty}d_{\rho,i}(e_i,h_i)=\lim_{i\to \infty}d_{\rho,i}(f_i,h_i)=\lim_{i\to \infty}d_{\rho,i}(f_i,g_i)=\lim_{i\to\infty}d_{\rho,i}(f_i,g_i).
        \]
    Consequently, \(d_{\rho}\) is well defined. 
    
    Next, we will show that \(d_\rho\) is indeed a metric. Let \([(e_i)_{i=1}^\infty],[(f_i)_{i=1}^\infty]\in X\). If \([(e_i)_{i=1}^\infty]=[(f_i)_{i=1}^\infty]\), then by definition \(e_i\cap f_i\neq \varnothing\) and hence \(d_{\rho,i}(e_i,f_i)=0\), for all \(i\in \N\), from which it follows that
        \[
            d_{\rho}([(e_i)_{i=1}^\infty],[(f_i)_{i=1}^\infty])=\lim_{i\to \infty}d_{\rho,i}(e_i,f_i)=0.
        \]
     If \([(e_i)_{i=1}^\infty]\neq[(f_i)_{i=1}^\infty]\), then there exists some \(j\in \N\) such that \(e_i\cap f_i=\varnothing\), and hence \(d_{\rho,i}(e_i,f_i)>0\), for all \(i\geq j\). Particularly, since the sequence \(d_{\rho,i}(e_i,f_i)\) is increasing by \ref{Lemma: DL(4)}, we have that \(d_{\rho,i}(e_i,f_i)>d_{\rho,i}(e_j,f_j)>0\) for all \(i\geq j\), and so it follows that
        \[
            d_{\rho}([(e_i)_{i=1}^\infty],[(f_i)_{i=1}^\infty])=\lim_{i\to \infty}d_{\rho,i}(e_i,f_i)\geq d_{\rho,i}(e_j,f_j)>0.
        \]
    Consequently, we have \(d_{\rho}([(e_i)_{i=1}^\infty],[(f_i)_{i=1}^\infty])=0\) if and only if \([(e_i)_{i=1}^\infty]=[(f_i)_{i=1}^\infty]\). Symmetry of \(d_\rho\) follows directly from the symmetry of the metrics \(d_{\rho,i}\). For the triangle inequality, \ref{Lemma: DL(3)} gives us that
        \[
            \lvert d_{\rho,i}(e_i,f_i)- d_{\rho,i}(e_i^-,f_i^-)\rvert\leq d_{\rho,i}(e_i^-,e_i^+)+d_{\rho,i}(f_i^-,f_i^+)\leq \rho_i(e_i)+\rho_i(f_i),
        \]
    and since \(\rho_i(e_i),\rho_i(f_i)\stackrel{i\to \infty}{\longrightarrow}0\), it follows that 
        \begin{equation}\label{Equation: Edge to Vertex}
            d_{\rho}([(e_i)_{i=1}^\infty],[(f_i)_{i=1}^\infty])=\lim_{i\to\infty}d_{\rho,i}(e_i,f_i)=\lim_{i\to\infty}d_{\rho,i}(e_i^-,f_i^-).
        \end{equation}
    Given a third point \([(h_i)_{i=1}^\infty]\in X\) we have that \(d_{\rho,i}(e_i^-,f_i^-)\leq d_{\rho,i}(e_i^-,h_i^-) + d_{\rho,i}(h_i^-,f_i^-)\) for all \(i\in \N\), from which it follows that
        \[
        \begin{split}
            d_\rho([(e_i)_{i=1}^\infty],[(f_i)_{i=1}^\infty])&\stackrel{\eqref{Equation: Edge to Vertex}}{=}\lim_{i\to\infty}d_{\rho,i}(e_i^-,f_i^-)\\
            &\stackrel{\phantom{\eqref{Equation: Edge to Vertex}}}{\leq} \lim_{i\to\infty} d_{\rho,i}(e_i^-,h_i^-)+\lim_{i\to\infty}d_{\rho,i}(h_i^-,f_i^-)\\
            &\stackrel{\eqref{Equation: Edge to Vertex}}{=}d_\rho([(e_i)_{i=1}^\infty],[(h_i)_{i=1}^\infty])+d_\rho([(h_i)_{i=1}^\infty],[(f_i)_{i=1}^\infty]),
        \end{split}
        \]
    and so \(d_\rho\) satisfies the triangle inequality.
\end{proof}

\begin{definition}\label{def:intrinsicmetric}
The intrinsic metric \(d_{L_*}\) on \(X\) is the metric \(d_\rho\) given by the trivial admissible density \(\rho\equiv L_*^{-1}\).
\end{definition}

While the construction of a metric $d_\rho$ works for all admissible $\rho$, we need the following symmetry condition which is used to establish the doubling property. It is also needed to ensure that the the metrics $d_\rho$ for different $\rho$ are quasisymmetric to each other.

\begin{definition}
    Let \(\rho\colon E_1\to (0,1)\) be \(\Theta^{(1)}\)-admissible, and suppose that the IGS is symmetric with an associated symmetry map \(\eta\colon G_1\to G_1\). We say that the density \(\rho\) is \emph{symmetric}, if
        \(
            \rho(\eta(e))=\rho(e)
        \)
    for all \(e\in E_1\).
\end{definition}

The limit space  has a covering $\mathbf{X}^n := \{X_e : e\in E_n\}$ for every $n\in \N$, and it satisfies the quasi-visual approximation condition in the sense of \cite[Definition 2.1]{bonkmeyer}. To prove this, we first derive some inequalities for the limit metric.

\begin{lemma}\label{Lemma: Distance Lemma 2}
    Let \(x=[(e_i)_{i=1}^\infty],y=[(f_i)_{i=1}^\infty]\in X\), and let \(\rho\colon E_1\to (0,1)\) be a \(\Theta^{(1)}\)-admissible density.
    \begin{enumerate}[label={\color{blue}{\textup{$($DL\arabic*}$)$}}, widest=a, leftmargin=*]\setcounter{enumi}{4}
        \item\label{Lemma: DL(5)} For all \(n\in \N\) we have that
            \[
                d_{\rho,n}(e_n,f_n)\leq d_\rho(x,y).
            \]
        \item\label{Lemma: DL(6)} If \(e_n\cap f_n\neq \varnothing\) for some \(n\in\N\), then
            \[
                d_{\rho}(x,y)\leq \rho_n(e_n)+\rho_n(f_n).
            \]
        \item\label{Lemma: DL(7)} If \(e_n\cap f_n\neq \varnothing\) for some \(n\in\N\) and the IGS and \(\rho\) are symmetric, then
            \[
                \frac{m_\rho}{M_\rho}\rho_n(e_n) \leq \rho_n(f_n) \leq \frac{M_\rho}{m_\rho}\rho_n(e_n).
            \]
        \item\label{Lemma: DL(8)} If \(e_n\cap f_n= \varnothing\) for some \(n\in\N\) and the IGS and \(\rho\) are symmetric, then
            \[
                d_\rho(x,y)\geq\frac{m_\rho}{M_\rho}\rho_n(e_n).
            \]
    \end{enumerate}
\end{lemma}
\begin{proof}
    We will begin with \ref{Lemma: DL(5)}. By \ref{Lemma: DL(4)}, we have  \(d_{\rho,i+1}(e_{i+1},f_{i+1})\geq d_{\rho,i}(e_i,f_i)\) for all \(i\in\N\), and so it follows for all \(n\in \N\) 
        \[
            d_{\rho}(x,y)=\lim_{i\to\infty}d_{\rho,i}(e_i,f_i)\geq d_{\rho,n}(e_n,f_n).
        \]

    \ref{Lemma: DL(6)} follows quite directly; since \(e_n\cap f_n\neq \varnothing\), the sub-graph \(e_n\cdot G_m\cup f_n\cdot G_m\) is connected and contains \(e_{n+m}\) and \(f_{n+m}\) for all \(m\in\N\). Therefore, \(d_{\rho,i}(e_i,f_i)\leq \rho_n(e_n)+\rho_n(f_n)\) for all \(i\geq n\), and so 
        \[
            d_\rho(x,y)=\lim_{i\to\infty}d_{\rho,i}(e_i,f_i)\leq \rho_n(e_n)+\rho_n(f_n).
        \]

     We move on to \ref{Lemma: DL(7)}. For any sequence \((h_i)_{i=1}^\infty\in \Sigma \) we have by construction that \((T^{-1}(h_n))_i=(T^{-1}((h_i)_{i=1}^\infty))_i\) for all \(1\leq i \leq n\). Thus, we may write \(\rho_n(e_n)\) and \(\rho_n(f_n)\) as the products
        \begin{equation}\label{Equation: Comparable edges 1}
            \rho_n(e_n)=\prod_{i=1}^n\rho\big(T^{-1}((e_i)_{i=1}^\infty)_i\big)\quad\text{and}\quad\rho_n(f_n)=\prod_{i=1}^n\rho\big(T^{-1}((f_i)_{i=1}^\infty)_i\big).
        \end{equation}
    Since the IGS is doubling, non-degenerate and symmetric, Lemma \ref{Lemma: Intersecting Sequences} ensures that there exists a number \(m\leq n\) so that
        \begin{equation}\label{Equation: Comparable edges 2}
            T^{-1}((e_i)_{i=1}^\infty)_i=
            \begin{cases}
                T^{-1}((f_i)_{i=1}^\infty)_i&\text{for } 1\leq i< m,\\
                T^{-1}((f_i)_{i=1}^\infty)_i \;\text{ or }\; \eta\circ T^{-1}((f_i)_{i=1}^\infty)_i&\text{for }m< i\leq n.
            \end{cases}
        \end{equation}
    As \(\rho\) is symmetric, we have \(\rho(T^{-1}((f_i)_{i=1}^\infty)_i)=\rho(\eta\circ T^{-1}((f_i)_{i=1}^\infty)_i)\) for all \(i\in\N\), and therefore it follows from \eqref{Equation: Comparable edges 2} that
        \begin{equation}\label{Equation: Comparable edges 3}
            \prod_{\substack{1\leq i\leq n\\ i\neq m }}\rho(T^{-1}((e_i)_{i=1}^\infty)_i)=\prod_{\substack{1\leq i\leq n\\ i\neq m }}\rho(T^{-1}((f_i)_{i=1}^\infty)_i).
        \end{equation}
    Combining \eqref{Equation: Comparable edges 1} and \eqref{Equation: Comparable edges 3}, we obtain
        \[
            \frac{\rho_n(e_n)}{\rho(T^{-1}((e_i)_{i=1}^\infty)_m)}=\frac{\rho_n(f_n)}{\rho(T^{-1}((f_i)_{i=1}^\infty)_m)},
        \]
    and since \(0<m_\rho\leq \rho(e)\leq  M_\rho<1\) for all \(e\in E_1\), we obtain the desired estimate
        \[
            \frac{m_\rho}{M_\rho}\rho_n(e_n)\leq \rho_n(f_n)\leq \frac{M_\rho}{m_\rho}\rho_n(e_n).
        \]

    Lastly, we prove \ref{Lemma: DL(8)}. Let \(v\in e_n\) and \(w\in f_n\) so that \(d_{\rho,n}(v,w)=d_{\rho,n}(e_n,f_n)\). Since \(e_n\cap f_n=\varnothing\), we have that \(v\neq w\), and may thus choose \(\Omega\in\Lambda^n_{v,w}\) so that 
        \[
            d_{\rho,n}(v,w)=\sum_{e\in \Omega}\rho_{\lvert e \rvert}(e).
        \]
    By definition, \(p_n(\Omega)\) contains a path joining \(v\) to \(w\), and therefore there exists an edge \(g\in p_n(\Omega)\) so that \(v\in g\cap e_n\). For every \(e\in\Omega\) for which \(g\in e\cdot G_{n-\lvert e \rvert}\), we have that \(\rho_n(g)\leq\rho_{\lvert e \rvert}(e)\) and therefore
        \[
            \rho_n(g)\leq\sum_{e\in \Omega}\rho_{\lvert e \rvert}(e)=d_{\rho,n}(v,w)=d_{\rho,n}(e_n,f_n).
        \]
    Since \(g\cap e_n\neq\varnothing\), we have by \ref{Lemma: DL(7)} that \(\rho_n(g)\geq (m_\rho/M_\rho)\cdot \rho_n(e)\). Due to \ref{Lemma: DL(5)}, we have that \(d_\rho(x,y)\geq d_{\rho,n}(e_n,f_n)\), and thus we obtain the desired inequality
        \[
            d_\rho(x,y)\geq d_{\rho,n}(e_n,f_n)\geq \rho_n(g)\geq\frac{m_\rho}{M_\rho}\rho_n(e_n).
        \]
\end{proof}

\begin{corollary}\label{Corollary: Diameter comparable to weight}
    If \(\rho\colon E_1\to (0,1)\) is a symmetric \(\Theta^{(1)}\)-admissible density, then
        \begin{equation}\label{Equation: Diameter comparable to weight}
            \frac{m_\rho^3}{M_\rho}\rho_n(e)\leq \diam (X_e,d_\rho)\leq \rho_n(e)\leq \frac{M_\rho}{m_\rho^3}\rho_n(e).
        \end{equation}
\end{corollary}
\begin{proof}
    Since every point \(x\in X_e\) has by definition a representative \((e_i)_{i=1}^\infty\in \Sigma\) for which \(e_n=e\), it follows from \ref{Lemma: DL(2)} that \(\diam(X_e,d_\rho)\leq \rho_n(e)\). Hence, it is enough to establish the leftmost inequality. By Lemma \ref{Lemma: Interior edge}, there exists \(f,h\in e\cdot E_2\) such that \(f\cap h=\varnothing\). Fix \(y\in X_f\) and \(z\in X_h\). Then, by definition, \(y\) has a representative \((f_i)_{i=1}^\infty\in \Sigma\) for which \(f_{n+2}=f\) and \(z\) has a representative \((h_i)_{i=1}^\infty\in \Sigma\) for which \(h_{n+2}=h\). Thus, since \(X_f,X_h\subset X_e\) and \(f\cap h =\varnothing\), we have by Lemma \ref{Lemma: DL(8)} that 
        \[
            \diam(X_e,d_\rho)\geq d_\rho(x,y)\geq \frac{m_\rho}{M_\rho}\rho_{n+2}(f)\geq\frac{m_\rho^3}{M_\rho}\rho_n(e),
        \]
    where the last inequality follows from equation \eqref{Equation: Max-min level bound}.
\end{proof}

Using the estimates presented thus far, we may establish some basic topological properties of the limit space.

\begin{lemma}\label{Lemma: Compact and Connected}
    For every \(\Theta^{(1)}\)-admissible density \(\rho\colon E_1\to (0,1)\), the metric space \((X,d_\rho)\) is compact and connected.

\end{lemma}
\begin{proof}
We will first prove that the canonical projection \(\chi\colon (\Sigma,\delta_\Sigma)\to(X,d_\rho)\) is continuous. If \(x=(e_i)_{i=1}^{\infty},y=(f_i)_{i=1}^\infty\in\Sigma\) so that \(\delta_\Sigma((e_i)_i^{\infty},(f_i)_i^\infty)<2^{-k}\) for some \(k\in\N\), then \(e_k\cap f_k\neq \varnothing\), and hence, due to \ref{Lemma: DL(6)},
    \[
        d_\rho(\chi(x),\chi(y))\leq\rho_k(e_k)+\rho_k(f_k)\leq 2M_\rho^k.
    \]
Consequently, for every \(\varepsilon>0\) one may choose \(\delta=2^{-\log(\varepsilon/4)/\log(M_\rho)}\), so that \(d_\rho(\chi(x),\chi(y))<\varepsilon\) whenever \(x,y\in\Sigma\) and \(\delta_\Sigma(x,y)<\delta\). Therefore, \(\chi\) is continuous. Since \((\Sigma,\delta_\Sigma)\) is compact and \(\chi\) is continuous,  \(\chi(\Sigma)=(X,d_\rho)\) is compact.

We next prove that \((X,d_\rho)\) is connected. For any \(x,y\in X\) and \(\varepsilon>0\), choose \(n\in\N\) so that \(M_\rho^n<\varepsilon\). Since \(G_n\) is connected, we may choose an edge path \(\theta=[v_0,e_1,v_2,\ldots,e_l,v_l]\) so that \(x\in X_{e_1}\), \(y\in X_{e_l}\) and \(l\leq \lvert E_n\rvert\). Choose a point \(z_i\in X_{e_i}\) for every \(1\leq i\leq l\). Since \(\diam(X_{e_i},d_\rho)\leq \rho_n(e_i)\leq M_\rho^n<\varepsilon\), we have that \(B_{d_\rho}(z_i,\varepsilon)\cap B_{d_\rho}(z_{i+1},\varepsilon) \neq \varnothing \) for every \(1\leq i\leq i-1\).  That is, every pair of points $x,y$ can be connected by an \emph{$2\varepsilon$-chain} $(z_i)_{i=1}^l$ (i.e. a collection of points s.t. $z_1=x,z_l =y$ and $d(z_i,z_{i+1})\leq 2\epsilon$ for each $i=1,\dots, l-1$). Since \(x,y\) and \(\varepsilon\) were arbitrary, and since $X$ is compact, it follows that \((X,d_\rho)\) is connected.
\end{proof}

\begin{remark}\label{rmk:quasiconvex}
    While the space \((X,d_\rho)\) need not be quasiconvex in general, we can recover quasiconvexity with relatively mild additional assumptions. One can show that if the IGS is symmetric and \(\rho\colon E_1\to (0,1)\) is a symmetric \(\Theta^{(1)}\)-admissible function, then the space \((X,d_\rho)\) is quasiconvex if and only if there exists a path \(\theta\in\Theta^{(1)}\) with \(L_\rho(\theta)=1\). For example, the intrinsic metric $d_{L_*}$ is quasiconvex whenever the IGS is symmetric.

    In the quasiconvex case, the metric $d_\rho$ has also another simpler construction: up to a bi-Lipschitz deformation it can be realized as the Gromov-Hausdorff limit of re-scaled path metrics on $G_n$, where edges are weighted by $\rho_n$. This gives in particular a slightly simpler description of the intrinsic metric $d_{L_*}$ in the symmetric case.

     We now briefly sketch the proof of quasiconvexity. This is not needed later in the paper, and is only included for the sake of completeness.

\end{remark}

    Since \(L_\rho(\theta)=1\), we have for any \(e\in E_n\) that \(L_{\rho_{n+1}}(e\cdot \theta)=\rho_n(e)\cdot L_\rho(\theta)=\rho_n(e)\). Thus, given any path \(\theta_n\) from \(v_a\) to \(v_b\) in \(G_n\), we may construct a path \(\theta_{n+1}\) in \(G_{n+1}\) from \(w_a\in\pi^{-1}_{n+1}(v_a)\) to \(w_b\in \pi^{-1}_{n+1}(v_b)\) as a concatenation, using symmetry in the same way as in the proof of Lemma \ref{Lemma: Subset Lemma}, so that \(L_{\rho_{n+1}}(\theta_{n+1})=L_{\rho_n}(\theta_n)\). Moreover, one can show that the path \(\theta_n\) can be promoted to a \(1\)-Lipschitz path \(\gamma\) on \(X\) with endpoints in \(X_{v_a}\) and \(X_{v_b}\), so that \(\ell(\gamma)\leq L_{\rho_n}(\theta_n)\).

    Now, for any \(x=[(e_i)_{i=1}^\infty]\), \(y=[(e_i)_{i=1}^\infty]\), one first fixes a path \(\theta_n\) joining an endpoint of \(e_n\) to an endpoint of \(f_n\), where \(n=\min\{n\in\N\mid e_n\cap f_n=\varnothing\}\). One then promotes it to an equal-length path in \(G_{n+1}\) and adds tail sections to obtain a path \(\theta_{n+1}\) joining \(e_{n+1}\) to \(f_{n+1}\). Continuing this process, at stage \(i>n\), the contribution of the tail sections to \(L_{\rho_i}(\theta_{i})\) is bounded above by \(K\cdot M_\rho^{i-n}\) for some \(K\geq 1\), and hence one obtains a sequence of \(1\)-lipschitz paths \((\gamma_i)_{i=n}^\infty\) from \(X_{e_i}\) to \(X_{f_i}\), so that \(\ell(\gamma_i)\leq L_{\rho_n}(\theta_n)+K\sum_{j=1}^\infty M_\rho^{j-n}\leq Cd(x,y)\) for all \(n\leq i<\infty\), for some constant \(C\geq 1\). As \(x=\bigcap_{i=1}^\infty X_{e_i}\) and \(y=\bigcap_{i=1}^\infty X_{f_i}\), the Arzelà–Ascoli theorem then quarantees the existance of a path \(\gamma\) from \(x\) to \(y\) on \(X\) with \(\ell(\gamma)\leq Cd(x,y)\).

\subsection{Gromov-Hausdorff convergence}
So far, the space $X$ was called a limit space for mostly intuitive reasons. In this subsection, which is not used in other parts of the paper, we prove that the limit space can also be obtained as a Gromov--Hausorff limit of replacement graphs; see e.g. \cite{bbi} for background on this notion. 

A sequence $(X_n, d_{X_n})$ \emph{Gromov-Hausdorff converges} to $(X,d)$ if for some sequence $\varepsilon_n \to 0$
\begin{itemize}
\item there exist mappings $\iota_n:X_n\to X$ s.t.
\[
d_{X_n}(x,y)-\varepsilon_n\leq d(\iota_n(x),\iota_n(y)) \leq d_{X_n}(x,y)+\varepsilon_n; 
\]
for all $x,y\in X_n$; and
\item for every $x\in X$ there exists an $x_n\in X_n$ s.t. $d(\iota_n(x_n),x)\leq \varepsilon_n$.
\end{itemize}

\begin{proposition}\label{Proposition: GH-convergence}
    For any \(\Theta^{(1)}\)-admissible density \(\rho\colon E_1\to (0,1)\), the sequence \( (G_n,d_{\rho,n}) _{n \in \N}\) Gromov-Hausdorff converges to the limit space \((X,d_{\rho})\).
\end{proposition}

\begin{proof}
    For all \(n\in\N\) and \(v\in V_n\), choose a vertex \(\hat{v}\in V_{n+1}\) so that \(\pi_{n+1}(\hat{v})=v\). Then, for every \(m>n\), define the mapping \(\iota_{n,m}\colon V_n\to V_{m}\) recursively, so that \(\iota_{n,n+1}(v):=\hat{v}\) for all \(n\in\N\) and \(\iota_{n,m}:= \iota_{n,n+1}\circ \iota_{n+1,n+2}\circ\ldots\circ \iota_{m-1,m}\). We then have that
        \[
            \iota_{n,m}\circ \iota_{m,k}=\iota_{m,k}\quad\text{and}\quad \pi_{k,m}\circ \iota_{n,k}=\iota_{n,m}\quad\text{for all }\;n<m<k.
        \]
    For all \(n\in\N\) and \(v\in V_n\), choose a sequence of edges \((e_{v,i})_{i=n}^\infty\) so that \(v\in e_{v,n}\), \(i_{n,m}(v)\in e_{v,m}\) and \(\pi_m(e_{v,m})=e_{v,m-1}\) for all \(m>n\), and define the mapping \(\iota_{n,m}^E\colon V_n\to E_m\) for all \(m\in\N\) by
        \[
            \iota_{n,m}^E(v)=\begin{cases}
                \pi_{n,m}(e_{v,n})  &\text{for }n>m,\\
                e_{v,m}             &\text{for }n\leq m.
            \end{cases}
        \]
    Then, for all \(n\in\N\), we may identify each \(v\in V_n\) with a point in \(X\) via the mapping \(\iota_{V_n}\colon V_n\to X\) defined by \(\iota_{V_n}(v):=[(\iota_{n,i}^E(v))_{i=1}^\infty]\). It follows from \ref{Lemma: DL(3)} and \ref{Lemma: DL(1)} for all \(v,w\in V_n\) that
        \[
        \begin{split}
            d_{\rho,m}(\iota_{n,m}(v),\iota_{n,m}(w)) - 2M_\rho^m   &\leq d_{\rho,m}(\iota_{n,m}^E(v),\iota_{n,m}^E(w))\\
            &\leq d_{\rho,m}(\iota_{n,m}(v),\iota_{n,m}(w)) = d_{\rho,n}(v,w)
        \end{split}
        \]
    for all \(m>n\), and hence
        \[
        \begin{split}
            d_{\rho}(\iota_{V_n}(v),\iota_{V_n}(w)) &=\lim_{m\to\infty}d_{\rho,m}(\iota_{n,m}^E(v),\iota_{n,m}^E(w))\\
        &=\lim_{m\to\infty}d_{\rho,m}(\iota_{n,m}(v),\iota_{n,m}(w))\\
        &=\lim_{m\to\infty}d_{\rho,n}(v,w)=d_{\rho,n}(v,w),
        \end{split}
        \]
    and so \(\iota_{V_n}\) is an isometric embedding, and the first condition of Gromov-Hausdorff convergence is satisfied. For every \(x= [(e_i)_{i=1}^\infty]\in X\) and \(n\in\N\) choose an endpoint \(v_n\in e_n\). It then follows from \ref{Lemma: DL(6)} that
        \[
            d_\rho(x,\iota_{V_n}(v_n))\leq 2\rho_n(e_n)\leq 2M_\rho^n,
        \]
    and so Gromov-Hausdorff convergence follows from $M_\rho \in (0,1)$.
\end{proof}

\subsection{Quasi-visual approximation}
    To verify the Ahlfors-regularity, we establish a good quasi-visual approximation for the limit space.

\begin{definition}[{\cite[Definition 2.1]{bonkmeyer}}]
    Let \(S\) be a bounded metric space. A sequence \(\{\mathbf{X}^n\}_{n\in\N_0}\) of finite covers of \(S\) by its subsets, so that \(\mathbf{X}^0=\{S\}\), is called a \emph{quasi-visual approximation} of \(S\) if there exits a constant $C$ s.t. the following properties are satisfied for all \(n\in\N_0\) with 
        \begin{enumerate}
            \item[(i)] \(\diam(X)\leq  C\diam(Y)\) for all \(X,Y\in \mathbf{X}^n\) with \(X\cap Y \neq \varnothing\).
            \vspace{2pt}
            \item[(ii)] \(\dist(X,Y)\geq C^{-1}\diam(X)\) for all \(X,Y\in \mathbf{X}^n\) with \(X\cap Y=\varnothing\).
            \vspace{2pt}
            \item[(iii)] \(C^{-1} \diam(Y)\leq \diam(X)\leq C \diam(Y)\) for all \(X\in \mathbf{X}^n,Y\in \mathbf{X}^{n+1}\) with \(X\cap Y\neq\varnothing\).
            \vspace{2pt}
            \item[(iv)] For some constants \(k_0\in \N\) and \(\lambda\in (0,1)\) independent of \(n\) we have \(\diam(Y)\leq \lambda \diam(X)\) for all \(X\in\mathbf{X}^n\) and \(Y\in \mathbf{X}^{n+k_0}\) with \(X\cap Y\neq \varnothing\).
        \end{enumerate}
\end{definition}

For limit spaces of IGS:s we use the coverings \(\{\mathbf{X}^n\}_{n\in\N_0}\), where \(\mathbf{X}^0:=\{X\}\) and \(\mathbf{X}^n:=\{X_e\mid e\in E_n\}\) for all \(n\in\N\). 

\begin{lemma}\label{Lemma: quasi-visual approximation}
    Let \(\rho\colon E_1\to (0,1)\) be a symmetric \(\Theta^{(1)}\)-admissible density. Then the sequence \(\{\mathbf{X}^n\}_{n\in\N_0}\), is a quasi-visual approximation of \((X,d_\rho)\).
\end{lemma}
\begin{proof}
    Properties (i) and (ii) follow directly from \ref{Lemma: DL(7)}, \ref{Lemma: DL(8)} and Corollary \ref{Corollary: Diameter comparable to weight}. Let \(X_e\in \mathbf{X}\) and \(X_f\in \mathbf{X}^{n+k}\) for some \(k\in \N\), so that \(X_e\cap X_f\neq\varnothing\). Then, there exists an edge \(g\in e\cdot E_k\) so that \(f\cap g=\varnothing\), and hence, by \ref{Lemma: DL(7)} and equation \eqref{Equation: Max-min level bound},
        \begin{equation}
            \frac{m_\rho^{k+1}}{M_\rho}\rho_n(e)\leq\frac{m_\rho}{M_\rho}\rho_{n+k}(g)\leq\rho_{n+k}(f)\leq\frac{M_\rho}{m_\rho}\rho_{n+k}(g)\leq \frac{M_\rho^{k+1}}{m_\rho}\rho_n(e).
        \end{equation}
    It then follows from Corollary \ref{Corollary: Diameter comparable to weight} that
        \begin{equation}\label{Equation: Quasi-visual proof 3 and 4}
            \frac{m_\rho^{k+4}}{M_\rho^2}\diam(X_e,d_\rho)\leq\diam(X_f,d_\rho)\leq\frac{M_\rho^{k+2}}{m_\rho^4}\diam(X_e,d_\rho).
        \end{equation}
    Since \(M_\rho<1\), there exists a value \(k_0\in \N\) such that \(M_\rho^{k+2}<m_\rho^4\). Therefore, properties (iii) and (iv) follow from equation \eqref{Equation: Quasi-visual proof 3 and 4} with choices \(k=1\) and \(k=k_0\), respectively.
\end{proof}

\subsection{Ahlfors regularity and quasisymmetry}
We now have all the tools to prove Theorem \ref{Main Theorem: Construction of Metrics}. This proof is given at the end of the section after some preparatory work.

Let $\rho:E_1\to (0,1)$ be a symmetric \(\Theta^{(1)}\)-admissible density. By the admissibility of $\rho$, we have
\[
    \mathcal{M}_1(\rho) = \sum_{e \in E_1} \rho(e) \geq 1.
\]
Since $M_\rho = \max_{e \in E_1} \rho(e) < 1$, there is a unique $Q \geq 1$ such that
        \[
            \cM_Q(\rho)=\sum_{e\in E_1}\rho(e)^Q=1.
        \]
        Indeed, this follows from the fact that the map $Q \mapsto \mathcal{M}_Q(\rho)$, that is continuous and strictly decreasing, vanishes at infinity and $\mathcal{M}_1(\rho) \geq 1$.
        
        Define a probability measure $\tau_1(\{e\})=\rho(e)^Q$ on $E_1$ and let $\tau$ be the product measure on $E_1^\N$ by the pushforward of $T$ as $\mathfrak{m}_{\rho}= T_*(\tau_\infty)$ on $\Sigma$, where $T$ was as defined in \ref{def:Tdef}. Note, that the pushforward of a measure $\nu$ on $X$ under a measurable map $f:X\to Y$ is given by $f_*(\nu)(A)=\nu(f^{-1}(A))$. 

    This measure satisfies  \(\mathfrak{m}_{\rho}(\Sigma_e)=\rho_{\lvert e \rvert }(e)^Q\) for $e\in E_n$. Next, let $\mu_\rho=\chi_*(\mathfrak{m}_{\rho})$. Observe that $\chi^{-1}(X_e)=\Sigma_e \bigcup_{f\cap e \neq \emptyset} \chi^{-1}(X_f\cap X_e)\cap \Sigma_f$. A simple calculation using Remark \ref{rmk:intersection} and $\mathfrak{m}_{\rho}(\chi^{-1}(X_{e,\pm}))=0$ shows then that $\mu_\rho(X_e) = \rho_n(e)^Q$ for all $n \in \N \cup \{0\}$ and $e \in E_n$; see \cite[Proposition 3.42]{ASEBShimizu} for more details.

\begin{lemma}\label{Lemma: Ahlfors Regular}
Let \(\rho\colon E_1\to(0,1)\) be a symmetric \(\Theta^{(1)}\)-admissible density. Then, the metric measure space \((X, d_\rho, \mu_\rho)\) is \(Q\)-Ahlfors regular, where \(Q\geq 1\) is the unique exponent satisfying
\[
\sum_{e\in E_1} \rho(e)^Q = 1.
\]
\end{lemma}

\begin{proof}
Fix \(x=[(e_i)_{i=1}^\infty]\in X\) and let \(0<r\leq 2\diam(X,d_\rho)\). We will start by showing that \(\mu_\rho(B(x,r))\gtrsim r^Q\). To that end, let \(n=\min\{i\in\N\mid r\geq \rho_i(e_i)\}\). We then have that \(X_{e_n}\subseteq B(x,r)\), and
    \begin{equation}\label{Equation: Regularity 1}
    \begin{cases}
        \rho_n(e_n)\leq r \leq 2\diam(X,d_\rho)\ &\text{if }n=1,\\
        \rho_n(e_n)\leq r < \rho_{n-1}(e_{n-1}) &\text{if }n>1.
    \end{cases}
    \end{equation}
Since \(\rho_i(e_i)\geq m_\rho \rho_{i-1}(e_{i-1})\) for all \(i>1\), and since \(\rho_1(e_1)\geq m_\rho\), it follows from equation \eqref{Equation: Regularity 1} that \(\rho_n(e_n)\geq r\cdot m_\rho/(2\diam(X,d_\rho))\), and so we approximate:
    \[
        \mu_\rho(B(x,r))\geq \mu_\rho(X_{e_n})=\rho_n(e_n)^Q\geq\frac{m_\rho^Qr^Q}{2^Q\diam(X,d_\rho)^Q}.
    \]  
Next, we show that \(\mu_\rho(B(x,r))\lesssim  r^Q\). Let \(k=\min\{i\in\N\mid r\geq \rho_i(e_i)\cdot m_\rho^2 /M_\rho\}\). If \(k=1\), then we have that \(r\geq \rho_1(e_1)\cdot  m_\rho^2/M_\rho\geq m_\rho^3 /M_\rho\), and hence \(r\cdot M_\rho/m_\rho^3\geq 1\), from which it follows that 
    \[
    \mu_\rho(B(x,r))\leq \mu_\rho(X)=1\leq \bigg(\frac{M_\rho }{m_\rho^3}\bigg)^Qr^Q.
    \]
On the other hand, if \(k>1\), then we have that
    \[
        \frac{m_\rho^2}{M_\rho}\rho_k(e_k)\leq r <\frac{m_\rho^2}{M_\rho}\rho_{k-1}(e_{k-1})\leq \frac{m_\rho}{M_\rho} \rho_k(e_k),
    \]
and so it follows from \ref{Lemma: DL(7)} and \ref{Lemma: DL(8)} that \(X_{f}\cap B(x,r)\neq \varnothing\) for some \(f\in E_k \) if and only if \(f\cap e_k\neq \varnothing\). Therefore, \(B(x,r)\subseteq\bigcup_{f\in E_k:f\cap e_k\neq\varnothing}X_f\), where we have by \ref{Lemma: DL(7)} that 
    \[
        \max_{f\in E_k:f\cap e_k\neq \varnothing}\rho_k(f)\leq \frac{M_\rho}{m_\rho}\rho_k(e_k)\leq \frac{M_\rho^2r}{m_\rho^3} \leq \frac{M_\rho r}{m_\rho^3}.
    \]
Now, since the IGS is doubling, there exists due to Lemma \ref{Lemma: Doubling degree bound} a constant \(C_{\deg}\in\N\), so that
    \[
        C_{\deg}=\max_{v\in V_1}\deg(v)=\sup_{n\in\N}\max_{v\in V_n}\deg(v).
    \] 
Therefore, we have that \(\lvert\{f\in E_k\mid f\cap e_k\neq \varnothing\}\rvert\leq 2C_{\deg}+1\), and thus may approximate:
    \[
    \begin{split}
    \mu_\rho(B(x,r))&\leq \sum_{f\in\Omega}\mu_\rho(X_f)\leq (2C_\degr +1)\bigg(\frac{M_\rho}{m_\rho}\rho_k(e_k)\bigg)^Q\\
    &\leq (2C_\degr +1)\bigg(\frac{M_\rho^2 r}{m_\rho^3}\bigg)^Q=(2C_\degr +1)\bigg(\frac{M_\rho^2}{m_\rho^3}\bigg)^Qr^Q.
    \end{split}
    \]
\end{proof}

\begin{corollary}
    The metric space \((X,d_{L_*})\) is \(\log_{L_*}(\lvert E_1\rvert)\)-Ahlfors regular.
\end{corollary}

The motivation to quasi-visual approximations in the present work is the following lemma that gives us a short proof of the quasisymmetry.

\begin{lemma}[{\cite[Proposition 2.7]{bonkmeyer}}]\label{Lemma: quasi-visual quasisymmetry}
    Let \((S,d)\) and \((T,\delta)\) be bounded metric spaces, the map \(F\colon S\to T\) be a bijection, and \(\{\textbf{X}^n\}_{n\in\N_0}\) be a quasi-visual approximation of \((S,d)\). Then \(F\colon S\to T\) is a quasisymmetry if and only if \(\{F(\textbf{X}^n)\}_{n\in\N_0}\) is a quasi-visual approximation of \((T,\delta)\).
\end{lemma}

We are now ready to prove Theorem \ref{Main Theorem: Construction of Metrics}.

\begin{proof}[Proof of Theorem \ref{Main Theorem: Construction of Metrics}]
    Let \(\rho,\nu\colon E_1 \to (0,1)\) be symmetric \(\Theta^{(1)}\)-admissible densities. By Lemma \ref{Lemma: Ahlfors Regular}, the spaces \((X,d_\rho,\mu_\rho)\) and \((X,d_\nu,\mu_\nu)\) are \(Q_\rho\)- and \(Q_\nu\)-Ahlfors regular, respectively, where \(Q_\rho\) and \(Q_\nu\) are the unique exponents for which
        \[
            \cM_{Q_\rho}(\rho)=\cM_{Q_\nu}(\nu)=1.
        \]
    By Lemma \ref{Lemma: Compact and Connected} both spaces are compact and thus bounded, and by Lemma \ref{Lemma: quasi-visual approximation} the collection \(\{\mathbf{X}^n\}_{n\in\N_0}\), where \(\mathbf{X}^0:=\{X\}\) and \(\mathbf{X}^n:=\{X_e\mid e\in E_n\}\) for all \(n\in\N\), is a quasi-visual approximation of both spaces. Therefore, it follows from Lemma \ref{Lemma: quasi-visual quasisymmetry} that \(\id \colon(X,d_\rho)\to(X,d_\nu)\) is a quasisymmetry.

\end{proof}

\section{Moduli and Conformal Dimension}\label{Section: Moduli and Conformal Dimension}

In this section, we characterize the value and attainment of the conformal dimension of a symmetric Laakso-type fractal space. For the rest of the section, we fix a simple IGS that is doubling, non-degenerate and symmetric, with a generator \(G_1=(V_1,E_1,\xi_1)\) and symmetry map by \(\eta\). We denote the limit space by \(X\).

\subsection{Critical exponent} 
Conformal dimension for general metric spaces can be studied via discrete moduli of graph approximations at different levels; see e.g. \cite{Carrasco,KL,Shaconf,MT}. In our case, this involves solving a moduli problem for the sequence of graphs $G_n$ approximating the limit space $X$. Individually solving these modulus problems is computationally hopeless because the sizes of these graphs grow exponentially with the level of approximation. The following lemma, however, shows that it actually suffices to solve the modulus problem for $G_1$, and the solutions for $G_n$ follow from a multiplicative construction. This is one of the most important aspects of our approach, and expresses the idea that $G_1$ contains all of the information needed to study the conformal geometry of $X$.

\begin{lemma}\label{Lemma: Optimal Density}
    For every \(n\in\N\) and \(p\in(1,\infty)\) the unique \(p\)-minimal \(\Theta^{(n)}\)-admissible density \(\rho_n\colon E_n\to[0,1]\) is symmetric, satisfies \(\rho_n(e)<1\) for all \(e\in E_n\), and admits a multiplicative cascade representation
    
        \begin{equation}\label{Equation: Optimal density is a cascade}
            \rho_n(e)=\prod_{i=1}^n\rho_1(T^{-1}(e)_i)\quad\text{for all }e\in E_n,
        \end{equation}
    where \(\rho_1\colon E_1\to [0,1)\) is the \(p\)-minimal \(\Theta^{(1)}\)-admissible density. Furthermore, the mapping \(t\mapsto\Mod_t(\Theta^{(n)},G_n)\) is continuous and strictly decreasing for all \(t\in[1,\infty)\).
\end{lemma}
\begin{proof}
    
    Fix \(p\in(1,\infty)\). Since the IGS is non-degenerate, we have that \(\len(\theta)\geq 2\) for all \(\theta\in\Theta^{(n)}\). It then follows directly from Lemma \ref{Lemma: Existence of Minimizers} that the unique \(p\)-minimal \(\Theta^{(n)}\)-admissible density \(\rho_n\colon E_n\to[0,1]\) satisfies \(\rho_n(e)<1\) for all \(e\in E_n\), and that the mapping \(t\mapsto\Mod_t(\Theta^{(n)},G_n)\) is continuous and strictly decreasing for all \(t\in[1,\infty)\). Since the IGS is symmetric, we have
        \[
            \cM_p(\rho_n\circ \eta)=\cM_p(\rho_n)=\Mod_p(\Theta^{(n)},G_n).
        \]
    Since the minimizer $\rho_n$ is unique, we necessarily have \(\rho_n\circ \eta=\rho_n\). 
    
    We now verify \eqref{Equation: Optimal density is a cascade} for each \(n\in\N\). First, by applying Proposition \ref{Proposition: Path lift} inductively, we see that
        \[
            \min_{\theta\in\Theta^{(n)}}L_{\rho_n}(\theta)\geq\big(\min_{\theta\in\Theta^{(1)}}L_{\rho}(\theta)\big)^n\geq1.
        \]
    Therefore, \(\rho_n\) is \(\Theta^{(n)}\)-admissible. Due to Proposition \ref{Proposition: Duality} and the uniqueness of \(\rho_n\), it is enough to show that there exists a unit flow \(\cJ_n\) from \(I^{(n)}_-\) to \(I^{(n)}_+\) such that, if \(\rho_n\) satisfies equation \eqref{Equation: Optimal density is a cascade}, then 
        \[
            \cM_p(\rho_n)^\frac{1}{p}\cdot \cE_q(\cJ_n)^\frac{1}{q}=1.
        \]
    Since \(\cM_p(\rho_1)=\Mod_p(\Theta^{(1)},G_1)\), there exists by Proposition \ref{Proposition: Duality} a unit flow \(\cJ_1\) from \(I_-\) to \(I_+\) so that \(\cM_p(\rho_1)^\frac{1}{p}\cdot\cE_q(\cJ_1)^\frac{1}{q}=1\). Note, that since \(\rho_1\circ \eta = \rho_1\), it follows that \(\cJ_1(\eta(z),\eta(e))=-\cJ_1(z,e)\) for every \(z\in e\in E_1\). For every \(n\in\N\), we define \(\cJ_{n+1}\) as follows: For every \(z\in e\in E_1\) and \(f\in E_n\) we set
        \[
            \cJ_{n+1}([z,f],(e,f)):=\cJ_{n}(f^-,f)\cJ_1(z,e).
        \]
    We will show by induction, that \(\cJ_{n+1}\) is a unit flow for from \(I_-^{(n+1)}\) to \(I_+^{(n+1)}\) for every \(n\in \N\). To that end, assume that \(\cJ_n\) is a unit flow. Since \(\cJ_1\) is unit flow, we have for all \(e\in E_1\) and \(f\in E_n\) that
        \[
        \begin{split}
            \cJ_{n+1}([e^+,f],(e,f))=\cJ_{n}(f^-,f)\cJ_1(e^+,e)
            &=-\cJ_{n}(f^-,f)\cJ_1(e^-,e)\\&=-\cJ_{n+1}([e^+,f],(e,f)).
        \end{split}
        \]
    Now, let \(v=[z,\hat{e}]\in e\in E_{n+1}\) for some \(z\in V_1\) and \(\hat{e}\in E_n\), and suppose that \(z\not\in I_-\cup I_+\). The edges with \(v\) as an endpoint are all contained in \(\hat{e}\cdot G_1\), and since \(\cJ_1\) is a flow, we have
        \[
        \begin{split}
            \divr(\cJ_{n+1})(v)=\sum_{\{f\mid v\in f\}}\cJ_{n+1}(v,f)&=\sum_{\{f\mid z\in f\}}\cJ_n(\hat{e}^-,\hat{e})\cJ_1(z,f)\\
            &=\cJ_n(\hat{e}^-,\hat{e})\cdot\divr(\cJ_1)(z) =0.
        \end{split}
        \]
    Suppose next that \(z\in I_-\cup I_+\). Then, \(v\in\pi_{n+1}^{-1}(\hat{v})\) for some \(\hat{v}=\pi_{n+1}(v)\in V_n\). Since the IGS is doubling, for every edge \(\hat{f}\in E_n\) with \(\hat{v}\) as an endpoint the set \(\hat{f}\cdot E_1\subset E_{n+1}\) contains exactly one edge \(f\) with \(v\) as an endpoint, and by construction no other edges with \(v\) as an endpoint can exist. Since the IGS is symmetric, there exists a vertex \(x\in I_-\) so that \(z\in\{x,\eta(x)\}\) and hence, for each of the edges \(\hat{f}\) with \(\hat{v}\) as an endpoint, the vertex \(v\) satisfies:
    \[
        v=[z,\hat{e}]=
        \begin{cases}
            [x,\hat{f}]&\text{if }\hat{v}=\hat{f}^-,\\
            [\eta(x),\hat{f}]&\text{if }\hat{v}=\hat{f}^+.
        \end{cases}
    \]
    Therefore, we have that
    \[
        \cJ_{n+1}(v,f):=
        \begin{cases}
            \cJ_n(\hat{v},\hat{f})\cJ_1(x,\mathfrak{e}(x))  &\text{if }\hat{v}=\hat{f}^-,\\
            -\cJ_n(\hat{v},\hat{f})\cJ_1(\eta(x),\mathfrak{e}(\eta(x))) &\text{if }\hat{v}=\hat{f}^+,
        \end{cases}
    \]
    and since \(\mathfrak{e}(\eta(x))=\eta(\mathfrak{e}(x))\) and \(\cJ_1(x,\mathfrak{e}(x))=-\cJ_1(\eta(x),\eta(\mathfrak{e}(x)))\), it then follows that
    \[
        \cJ_{n+1}(v,f)=\cJ_n(\hat{v},\hat{f})\cJ_1(x,\mathfrak{e}(x)),
    \]
    from which we may compute
    \begin{equation}\label{Equation: Divergence Lift}
    \begin{split}
        \divr(\cJ_{n+1})(v)=\sum_{\{f\mid v\in f\}}\cJ_{n+1}(v,f)&=\sum_{\{\hat{f}\mid \hat{v}\in\hat{f}\}}\cJ_n(\hat{v},\hat{f}) \cJ_1(x,\mathfrak{e}(x))\\
        &=\cJ_1(x,\mathfrak{e}(x))\cdot\divr(\cJ_n)(\hat{v}).
    \end{split}
    \end{equation}
    If \(v\not\in I_-^{(n+1)}\cup I_+^{(n+1)}\), then \(\hat{v}\not\in I_-^{(n)}\cup I_+^{(n)}\), and consequently \(\divr(\cJ_n)(\hat{v})=0\), and so we have by \eqref{Equation: Divergence Lift} that \(\divr(\cJ_{n+1})(v)=0\). Hence, \(\cJ_{n+1}\) is a flow. To see that it is a unit flow, we compute inductively the total flow:
    \[
    \begin{split}
        \sum_{v\in I_-^{(n+1)}}\divr(\cJ_{n+1})(v)&\stackrel{\eqref{Equation: Divergence Lift}}{=}\sum_{\hat{v}\in I_-^{(n)}}\sum_{x\in I_-^{(1)}}\cJ_1(x,\mathfrak{e}(x))\cdot \divr(\cJ_n)(\hat{v})\\
        &\stackrel{\phantom{\eqref{Equation: Divergence Lift}}}{=}\sum_{\hat{v}\in I_-^{(n)}}\divr(\cJ_n)(\hat{v})\sum_{x\in I_-^{(1)}}\cJ_1(x,\mathfrak{e}(x))\\
        &\stackrel{\phantom{\eqref{Equation: Divergence Lift}}}{=}\sum_{\hat{v}\in I_-^{(n)}}\divr(\cJ_n)(\hat{v})=1
    \end{split}
    \]
    Having constructed the unit flow \(\cJ_n\) from \(I^{(n)}_-\) to \(I^{(n)}_+\) for every \(n\in\N\), we only need show that if \(\rho_n\) satisfies equation \eqref{Equation: Optimal density is a cascade} then \(\cM_p(\rho_n)^{1/p} \cdot \cE_q(\cJ_n)^{1/q}=1\). To that end, first note that for every \(e\in E_{n}\) we have by construction that
    \begin{equation}\label{Equation: Optimal flow is a cascade}
        \lvert \cJ_{n}(e)\rvert=\prod_{i=1}^{n}\lvert \cJ_1(T^{-1}(e)_i)\rvert.
    \end{equation}
    Then, if \(\rho_n\) satisfies equation \eqref{Equation: Optimal density is a cascade}, we may compute:
    \begin{equation}\label{Equation: Energy and mass cascade}
    \begin{split}
        \cE_q(\cJ_n)&\stackrel{\eqref{Equation: Optimal flow is a cascade}}{=}\sum_{e\in E_n}\prod_{i=1}^{n}\lvert \cJ_1(T^{-1}(e)_i)\rvert^q=\prod_{i=1}^{n}\sum_{e\in E_1}\lvert\cJ_1(e)\rvert^q=\cE_q(\cJ_1)^n,\\
        \cM_p(\rho_n)&\stackrel{\eqref{Equation: Optimal density is a cascade}}{=}\sum_{e\in E_n}\prod_{i=1}^n\lvert \rho_1(T^{-1}(e)_i)\rvert^p=\prod_{i=1}^n\sum_{e\in E_1}\lvert\rho_n(e)\rvert^p=\cM_p(\rho_1)^n.
    \end{split}
    \end{equation}
    Since \(\cM_p(\rho_1)^{1/p} \cdot \cE_q(\cJ_1)^{1/q}=1\), it then follows that
    \[
        \cM_p(\rho_n)^\frac{1}{p}\cdot \cE_q(\cJ_n)^\frac{1}{q}\stackrel{\eqref{Equation: Energy and mass cascade}}{=}\cM_p(\rho_1)^\frac{n}{p}\cdot \cE_q(\cJ_1)^\frac{n}{q}=1.
    \]  
\end{proof}

As a direct corollary of equations \eqref{Equation: Optimal density is a cascade} and \eqref{Equation: Energy and mass cascade}, we obtain the following explicit formula for the moduli at different levels. 

\begin{corollary}\label{Corollary: Modulus in multiplicative}
    For any \(p\in(1,\infty)\), we have for all \(n\in \N\) that 
        \[
            \Mod_p(\Theta^{(n)},G_n)=\Mod_p(\Theta^{(1)},G_1)^n.
        \]
\end{corollary}

The critical exponent $Q_*$ is obtained by solving the equation $\Mod_{Q_*}(\Theta^{(1)}, G_1)=1$. Solutions to this equation are found by the intermediate value theorem, since its fairly easy to see that $\Mod_{1}(\Theta^{(1)}, G_1)\geq 1$, since there is at least one path in $\Theta^{(1)}$. In the next lemma we characterize the case $\Mod_{1}(\Theta^{(1)}, G_1)=1$.  

\begin{lemma}\label{Lemma: Cut-points}
    The following four conditions are equivalent:
    \begin{enumerate}
        \item[(A)] \(\Mod_1(\Theta^{(1)},G_1)=1\).
        \item[(B)] There exists an edge \(e^*\in E_1\) that is contained in every path \(\theta\in\Theta^{(1)}\).
        \item[(C)] There are no edge disjoint paths in \(\Theta^{(1)}\).

    \end{enumerate}
\end{lemma}
\begin{proof}
    First, we prove that \((A)\Longrightarrow(C)\).
    Suppose \(\Mod_1(\Theta^{(1)},G_1)=1\). 
    Then, there exists a \(\Theta^{(1)}\)-admissible density \(\rho\colon E_1\to[0,1]\), for which \(\cM_1(\rho)=1\). If there is a simple path \(\theta\in\Theta^{(1)}\) which does not contain every edge \(e\in \supp(\rho)\), then
        \[
            L_\rho(\theta)=\sum_{e\in \theta}\rho(e)<\sum_{e\in E_1}\rho(e)=\cM_1(\rho)=1,
        \]
    which contradicts the admissibility of \(\rho\). Therefore, every path in \(\Theta^{(1)}\) contains all of the edges in \(\supp(\rho)\). In particular now two edge disjoint paths exist in $\Theta^{(1)}$. 
    
    The implication \((C)\Longrightarrow(B)\) is a direct consequence of Menger's Theorem \cite{menger}. 
    
    Finally \((B)\Longrightarrow(A)\) follows by setting $\rho(e^*)=1$ and $\rho(e)=0$ for all $e\neq e^*$ and observing that $\rho$ is $\Theta^{(1)}$-admissible.
\end{proof}

\begin{remark}
    If $\Mod_1(\Theta^{(1)},G_1)=1$, it is not difficult to conclude from the previous lemma that there exist a cut-point $p\in X$, i.e. points $p$ s.t. $X\setminus \{p\}$ is disconnected. Moreover, this point can be chosen such that $p$ disconnects $X_-$ from $X_+$. Also the converse of this is true, but a bit involved to show. In Theorem \ref{Main Theorem: Construction of Metrics}, we will also show that all these conditions are equivalent with the limit space having conformal dimension one. In short: the presence of cut points is equivalent to conformal dimension being $1$, and the lack thereof with conformal dimension being greater than $1$. These should be compared with \cite[Theorem 1.2]{carrascosplitting} and \cite[Corollary 1.2]{mckaydimone}, which are general statements on the connection between  conformal dimension and the presence of cut points.
\end{remark}

\begin{lemma}\label{Lemma: Critical exponent}
    There exists a unique value \(Q_*\geq 1\) for which \(\Mod_{Q_*}(\Theta^{(1)},G_1)=1\). Furthermore, the value \(Q_*\) satisfies
        \[
        \begin{split}
            Q_*&=\inf
            \left\{
                Q\geq1\ \middle\vert
                \begin{array}{l}
                    \exists \rho\colon E_1\to(0,1) \text{ s.t. } \cM_Q(\rho)=1,\\ \rho \text{ is symmetric and }\Theta^{(1)}\text{-admissible}
                \end{array}
            \right\}\\
            &=\inf\left\{ Q\geq 1\ \middle\vert \lim_{n\to\infty}\Mod_Q(\Theta^{(n)},G_n)=0\right\}.
        \end{split}
        \]
\end{lemma}

\begin{proof}
    Let $p\in (1,\infty)$ and let \(\rho\colon E_1\to [0,1)\) be the \(p\)-minimal \(\Theta^{(1)}\)-admissible density for some \(p\in(1,\infty)\). We then have by Lemma \ref{Lemma: Optimal Density} that \(\rho(e)<1\) for all \(e\in E_1\), and therefore
        \[
            \lim_{t\to\infty}\Mod_t(\Theta^{(1)},G_1)\leq \lim_{t\to\infty}\cM_t(\rho)=0.
        \]
    Since we have that \(\Mod_1(\Theta^{(1)},G_1)\geq 1\) due to the admissibility condition, and since the mapping \(t\to\Mod_t(\Theta^{(1)},G_1)\) is continuous and strictly decreasing for all \(t\in [1,\infty)\) by Lemma \ref{Lemma: Optimal Density}, it follows that there exists a unique value \(Q_*\geq 1\), such that \(\Mod_{Q_*}(\Theta^{(1)},G_1)=1\). 
    
    Next, we will show that \(Q_*=Q_{\inf}\), where
        \[
            Q_{\inf}:=\inf
            \left\{
                Q\geq1\ \middle\vert
                \begin{array}{l}
                    \exists \rho\colon E_1\to(0,1) \text{ s.t. } \cM_Q(\rho)=1,\\ \rho \text{ is symmetric and }\Theta^{(1)}\text{-admissible}
                \end{array}
            \right\}.
        \]
    Let \(\varepsilon>0\). By Lemma \ref{Lemma: Optimal Density}, there exists a symmetric \(\Theta^{(1)}\)-admissible density \(\rho_\varepsilon\colon E_1\to[0,1)\), so that \(\cM_{Q_*+\varepsilon}(\rho_\varepsilon)<1\). Thus, for some value \(\delta\in(0,1-\cM_{\rho_\varepsilon})\), the symmetric \(\Theta^{(1)}\)-admissible density \(\nu_\varepsilon\colon E_1\to(0,1)\) given by 
        \[
            \nu_\varepsilon(e):=\rho_\varepsilon(e)+\delta\quad\text{for all }e\in E_1
        \]
    satisfies \(\cM_{Q_*+\varepsilon}(\nu_\varepsilon)=1\). Since \(\varepsilon>0\) was arbitrary, it follows that \(Q_{\inf}\leq Q_*\). 
    
    On the other hand, since we have by Lemma \ref{Lemma: Optimal Density} that the map \(t\mapsto\Mod_t(\Theta^{(1)},G_1)\) is strictly decreasing, it follows that
        \[
            \cM_Q(\rho)\geq\Mod_Q(\Theta^{(1)},G_1)\geq \Mod_{Q_*}(\Theta^{(1)},G_1)=1
        \]
    whenever \(1\leq Q\leq Q_*\) and \(\rho\colon E_1\to (0,1)\) is a symmetric \(\Theta^{(1)}\)-admissible density, and therefore we have that \(Q_*\leq Q_{\inf}\).
    
    Finally, we show that \( Q_*=\inf\{Q\geq 1\mid \lim_{n\to\infty}\Mod_Q(\Theta^{(n)},G_n)=0\}\). Using again the fact that \(t\to\Mod_t(\Theta^{(1)},G_1)\) is strictly decreasing, it follows from Corollary \ref{Corollary: Modulus in multiplicative} that for all \(Q\in[1,\infty)\) we have
        \[
            \lim_{n\to\infty}\Mod_Q(\Theta^{(n)},G_n)=\begin{cases}
                \infty&\text{if }Q<Q_*,\\
                1&\text{if }Q=Q_*,\\
                0&\text{if }Q>Q_*.
            \end{cases}
        \]
    The claim now follows directly.
\end{proof}

We now obtain the critical exponent \(Q_*\) as an upper bound for the conformal dimension.

\begin{corollary}\label{Corollary: Confdim upper bound}
    Let \(Q_*\geq1\) be the unique value for which \(\Mod_{Q_*}(\Theta^{(1)},G_1)=1\). Then \(\dim_\AR(X,d_{L_*})\leq Q_*\).
\end{corollary}
\begin{proof}
    By Lemma \ref{Lemma: Critical exponent}, there exists a sequence of symmetric \(\Theta^{(1)}\)-admissible densities \((\rho_n)_{n=1}^\infty\colon E_1\to (0,1)\) and a sequence \((Q_n)_{n=1}^\infty\in (Q_*,\infty)\) with \(\inf_{n\in\N}Q_n=Q_*\), so that \(\cM_{Q_n}(\rho_n)=1\) for all \(n\in\N\). Therefore, we have by Theorem \ref{Main Theorem: Construction of Metrics} that \(d_{\rho_n}\in\cG_\AR(X,d_{L_*})\) for every \(n\in\N\) and thus
        \begin{equation}\label{Equation: Confdim upper bound}
            \dim_\AR(X,d_{L_*})\leq\inf_{n\in\N}\dim_{\mathrm{H}}(X,d_{\rho_n})=\inf_{n\in\N}=Q_*.
        \end{equation}
\end{proof}

Next, we prove the corresponding lower bound for conformal dimension. This will rely on finding appropriate porous subsets.

\subsection{Porous subsets} 
Let \(F\subset E_1\). We define the subsets \(E_n(F)\subset E_n\) and \(V_n(F)\subset V_n\) for every \(n\in \N\) by
        \[
        \begin{split}
            E_n(F)&:=\{e_1\cdot\ldots\cdot e_n\in E_n\mid e_i\in F\text{ for all }1\leq i\leq n\},\\
            V_n(F)&:= \{v\in \{e^-,e^+\}\mid e\in E_n(F)\},
        \end{split}
        \]
    and denote by \(G_n(F)\) the sub-graph of \(G_n\) with edges \(E_n(F)\) and vertices \(V_n(F)\). We define the subset \(X(F)\subset X\) as
        \[
            X(F):=\bigcap_{n=1}^\infty \bigcup_{e\in E_n(F)}X_e=\{[(e_n)_{n=1}^\infty]\in X\mid e_n\in E_n(F)\text{ for all }n\in\N \},
        \]
    and for every \(n\in\N\) and \(e\in E_n(F)\) the subset \(X_e(F):=X_e\cap X(F)\).

The following lemma is needed to study the geometry of the subset $X(F)$. 

\begin{lemma}\label{Lemma: Subset Lemma}
    Let \(F\subseteq E_1\) so that \(G_1(F)\) contains a path \(\theta=[v_0,e_1,\ldots,e_k,v_k]\) for which \(v_k=\eta(v_0)\). Then 
        \begin{enumerate}
            \item\label{Lemma: Subset Lemma (1)} \(G_n(F)\) contains a path in \(\Theta^{(n)}\) for every \(n\in\N\).
            \item\label{Lemma: Subset Lemma (2)} For every \(e\in E_n(F)\), there exists a point \(x\in \mathrm{int}(X_e(F))\) such that
                \[
                    B_{d_{L_*}\vert_{{X(F)}}}(x,L_*^{-(n+2)})\subset X_e(F).
                \]
            \item\label{Lemma: Subset Lemma (3)} \(X(F)\) is complete and uniformly perfect.
            \item\label{Lemma: Subset Lemma (4)} For every \(\delta\in\cG_\AR(X(F),d_{L_*})\) there exists a constant \(K>1\) such that for every \(n\in\N\) and \(e\in E_n(F)\) there exists \(x\in X_e(F)\) and \(r>0\) so that
                \[
                    B_\delta(x,r)\subset X_e(F)\subset B_\delta(x,Kr).
                \]
        \end{enumerate}
\end{lemma}

\begin{proof}
    First, we will prove claim \eqref{Lemma: Subset Lemma (1)} by induction. Suppose \(G_n(F)\) contains a path \([u_0,f_1,\ldots,f_m,u_m]\in \Theta^{(n)}\). Since \(v_0\in I_-\) and \(v_k=\eta(v_0)\in I_+\), it follows for all \(1\leq i\leq m-1\) that
        \begin{equation}\label{Equation: Symmetric concatenation}
        \begin{cases}
            (v_k,f_i)\sim (v_0,f_{i+1})&\text{if and only if}\quad u_i=f_i^+=f_{i+1}^-,\\
            (v_k,f_i)\sim (v_k,f_{i+1})&\text{if and only if}\quad u_i=f_i^+=f_{i+1}^+,\\
            (v_0,f_i)\sim (v_0,f_{i+1})&\text{if and only if}\quad u_i=f_i^-=f_{i+1}^-,\\
            (v_0,f_i)\sim (v_k,f_{i+1})&\text{if and only if}\quad u_i=f_i^-=f_{i+1}^+.
        \end{cases}
        \end{equation}
    Define for all \(1\leq i\leq m\) the path
        \[
            \theta_i:=[v^i_0,e^i_1,\ldots,e^i_k,v^i_k]=\begin{cases}
                \theta&\text{if}\quad u_i=f_i^+,\\
                \theta^{-1}&\text{if}\quad u_i=f_i^-.
            \end{cases}
        \]
    Then, by \eqref{Equation: Symmetric concatenation}, we have for all \(1\leq i\leq m-1\) that \([v^i_k,f_i]=[v^{i+1}_0,f_i]\), and hence the concatenation \(\vartheta:=(f_1\cdot\theta_i)* (f_2\cdot\theta_2)*\ldots *(f_m\cdot\theta_m) \) is well defined. Moreover, we have that \([v^1_0,f_1]\in I^{(n+1)}_-\) and \([v^m_k,f_m]\in I^{(n+1)}_+\), and therefore \(\vartheta\in\Theta^{(n+1)}\). Since \(f_i\in E_n(F)\) and \(e_j\in E_1(F)\), we have by definition that \(f_i\cdot e_j\in E_{n+1}(F)\) for all \(1\leq i\leq m\) and \(1\leq j\leq k\), and so the path \(\hat{\theta}\) is contained in \(G_{n+1}(F)\). This concludes the induction.
    \vspace{5pt}

    We will now prove claim \eqref{Lemma: Subset Lemma (2)}. By claim \eqref{Lemma: Subset Lemma (1)}, there exists a path \(\theta\in\Theta^{(2)}\) with every edge contained in \(E_2(F)\). Since the IGS is non-degenerate, we have that \(\len(\theta)\geq 4\), and may thus choose an edge \(h\in E_2(F)\) so that \(h\cap I_{\pm}^{\smash{(2)}}=\varnothing\). Consequently, we may choose an edge \(f=(e\cdot h)\in e\cdot E_2(F)\) so that \(f\cap g=\varnothing\) for all \(g\in E_{n+2}(F)\setminus (e\cdot E_2(F))\). By definition, for any \(x\in X_f(F)\) and \(y\in X(F)\setminus X_e(F)\), there exists representatives \((e_i)_{i=1}^\infty\) and \((g_i)_{i=1}^\infty\), respectively, so that \(e_{n+2}=f\) and \(g_{n+2}\not\in e\cdot E_2(F)\) and therefore \(e_{n+2}\cap f_{n+2}=\varnothing\). It then follows from \ref{Lemma: DL(8)} that \(d_{L_*}(x,y)\geq L_*^{-n-2}\) whenever \(x\in X_f(F) \) and \(y\in X(F)\setminus X_e(F)\), and therefore
        \[
            x\in B_{d_{L_*}\vert_{X(F)}}(x,L_*^{-n-2})\subset \mathrm{int}(X_e(F)).
        \]

    Next, we prove claim \eqref{Lemma: Subset Lemma (3)}, beginning with completeness. Since the set \(X_e\) is closed for any \(e\in E_n\), the finite union \(\bigcup_{e\in E_n(F)}X_e\) is closed for all \(n\in\N\). Since arbitrary intersections of closed sets are closed, \(X(F)\) is a closed subset of a compact metric space. Thus \(X(F)\) is complete. For uniform perfectness, let \(C=2L_*^3\diam(X(F),d_{L_*})\) and \(x=[(e_i)_{i=1}^\infty]\in X(F)\). By claim \eqref{Lemma: Subset Lemma (1)}, there exists a path \(\theta\in \Theta^{(2)}\) with every edge contained in \(E_2(F)\). Since the IGS is non-degenerate, we have that \(\len(\theta)\geq 4\), and may thus choose for any edge \(e\in E_2(F)\) an edge \(f\in E_2(F)\) so that \(e\cap f=\varnothing\). Consequently, for each \(n\in\N\), we may choose an edge \(f_{n+2}\in e_n\cdot E_2(F)\) so that \(e_{n+2}\cap f_{n+2}=\varnothing\), and a point \(y_n\in X_{f_{n+2}}(F)\). It then follows that
        \[
            L_*^{-n-2}\leq d_{L_*}(x,y_n)\leq L_*^n
        \]
    for every \(n\in \N\). Let \(r\in (0,\diam(X(F),d_{L_*})\) and 
        \[
            k:=\max\{k\in\N\mid 2L_*<r\leq 2L_*^{-k+1}\diam(X(F),d_{L_*})\}.
        \]
    We then have that \(y_k\in B(x,2L_*^{-k})\setminus B(x,L_*^{-k-2})\subset B(x,r)\setminus B(x,r/C)\) and there \(B(x,r)\setminus B(x,r/C)\neq \varnothing\). Since \(x\in X(F)\) and \(r\in (0,\diam(X(F),d_{L_*}))\) were arbitrary, it follows that \((X(F),d_{L_*})\) is \(C\)-uniformly perfect.
    \vspace{5pt}

    Finally, we prove claim \eqref{Lemma: Subset Lemma (4)}. By claim \eqref{Lemma: Subset Lemma (3)}, the metric space \((X(F),d_{L_*})\) is complete and uniformly perfect, and hence a metric \(\delta\in\cG_\AR(X(F),d_{L_*})\) exists. By claim \eqref{Lemma: Subset Lemma (2)}, there exists a point \(x\in X_e(F)\) so that
        \[
            B_{d_{L_*}\vert_{X(F)}}(x,L_*^{-n-2})\subset X_e(F).
        \]
    Let \(\eta\) be the distortion function of the quasisymmetry \(\id\colon (X(F),d_{L_*})\to (X(F),\delta)\) and let \(K=4\eta(2L_*^2)\). We may assume that \(\diam(X_e(F),\delta)>0\), since otherwise \(\diam(X(F),\delta)=0\) and the claim becomes trivial. Let \(r=2K^{-1}\diam(X_e(F),\delta)\). Then, since \(\diam(X_e(F),d_{L_*})\leq L_*^{-n}\), it follows that \(X_e(F)\subset B_\delta(x,Kr)\), and so it suffices to show that \(B_\delta(x,r)\subset X_e(F)\). To that end, let \(y\in B_\delta(x,r)\) and pick \(z\in X_e(F)\) with \(\delta(x,z)\geq\diam(X_e(F),\delta)/2=Kr/4\). By the quasisymmetry condition
        \[
            \frac{K}{4}\leq\frac{\delta(x,z)}{\delta(x,y)}\leq \eta\bigg(\frac{d_{L_*}(x,z)}{d_{L_*}(x,y)}\bigg).
        \]
    Thus,
        \[
            \frac{d_{L_*}(x,y)}{d_{L_*}(x,z)}\leq\frac{1}{\eta^{-1}(K/4)},
        \]
    and
        \[
            d_{L_*}(x,y)\leq\frac{d_{L_*}(x,z)}{\eta^{-1}(K/4)}\leq\frac{L_*^{-n}}{\eta^{-1}(K/4)}<L_*^{-n-2}.
        \]
    Hence, \(y\in X_e(F)\). As \(y\in B_\delta(x,r)\) was arbitrary, it follows that  \(B_\delta(x,r)\subset X_e(F)\) as claimed.
\end{proof}

    With the aim of using porosity to obstruct attainment of conformal dimension, we establish the following.

\begin{lemma}\label{Lemma: Porous}
    If \(F\subsetneq E_1\) is non-empty, then \(X(F)\) is a closed porous subset of \((X,d_{L_*})\).
\end{lemma}
\begin{proof} 
    Fix a constant \(c=L_*^{-5}\cdot\diam(X,d_{L_*})^{-1}\). Let \(y=[(e_i)_{i=1}^\infty]\in X(F)\), \(0<r\leq \diam(X,d_{L_*})\) and choose \(n\in\N\) so that
        \[
            L_*^{-n}\leq r\leq L_*^{-(n-1)}\cdot \diam(X,d_{L_*}).
        \]
    Since \(\diam(X_{e_{n+1}},d_{L_*})\leq L_*^{-(n+1)}\), it follows that \(X_{e_{n+1}}\subset B(y,L_*^{-n})\subset B(y,r)\). Now, choose any \(g\in E_1\setminus F\) and let \(f=e_{n+1}\cdot g\in E_{n+2}\). Then, by definition, \(f\not\in E_{n+2}(F)\) and therefore \(\mathrm{int}(X_f)\subset X\setminus X(F)\), from which it follows that
        \[
            \mathrm{int}(X_f)\subset X_{e_{n+1}}\setminus X(F)\subset B(y,r)\setminus X(F).
        \]
    Note that \(X=X(E_1)\). Thus, by Lemma \ref{Lemma: Subset Lemma}, there exists \(x\in \mathrm{int}(X_f)\), so that \(B(x,L_*^{-(\lvert f\rvert+2)})=B(x,L_*^{-(n+4)})\subset \mathrm{int}(X_f)\). Since \(r\leq L_*^{-(n-1)}\cdot \diam(X,d_{L_*})\), the constant \(c\) then satisfies 
        \[
            B(x,cr)\subset B(x,L^{-(n+4)})\subset\mathrm{int}(X_f)\subset B(y,r)\setminus X(F).
        \]
    Since the point \(y\in X(F)\) and the radius \(0<r\leq \diam(X,d_{L_*})\) were arbitrary, we conclude that \(X(F)\) is a porous subset of \((X,d_{L_*})\).
\end{proof}

\subsection{Characterization of attainment}
    For this subsection, we fix a simple path
        \[
            \theta^*:=[v^*_0,e^*_1,\ldots,e^*_{L_*},v^*_{L_*}]\in \Theta^{(1)}
        \]
    for which \(\eta(v^*_0)=v^*_k\) and \(\len(\theta^*)=L_*\). If $Q_*>1$, we also fix $\rho^*$, which is the minimizer for $\Mod_{Q_*}(\Theta^{(1)}, G_1).$
    
\begin{definition}\label{def:removable}
    We define the subset \(\widehat{E}\subseteq E_1\) by    
        \[
            \widehat{E}:=\begin{cases}
                \{e : e\in \theta^*\}&\text{if}\quad Q_*=1,\\
                \supp(\rho^*)&\text{if}\quad Q_*>1,
            \end{cases}
        \]
    and define for all \(n\in\N\) the family of paths 
        \[
            \widehat{\Theta}^{(n)}:=\{[v_0,e_1,\ldots,e_k,v_k]\in \Theta^{(n)}\mid e_i\in E_n(\widehat{E}) \text{ for all }1\leq i\leq k\}.
        \]
    Furthermore, we denote \(\widehat{X}:=X(\widehat{E})\), \(\widehat{X}_\pm:=X_\pm\cap \widehat{X}\), and \(\widehat{X}_e:=X_e(\widehat{E})\). We say that the IGS has a removable edges if \(\widehat{E}\neq E_1\). 
\end{definition}
The edges in \(\widehat{E}\neq E_1\) are called removable, since removing them from $E_1$ does not alter the value of modulus, as the following lemma shows. Recall that $\Mod_{Q_*}(\Theta^{(n)}, G_n)=1$.

\begin{lemma}\label{Lemma: Critical exponent (porous)}
    For every \(n\in\N\), the family of paths \(\widehat{\Theta}^{(n)}\) is non-empty and satisfies \(\Mod_{Q_*}(\widehat{\Theta}^{(n)},G_n(\widehat{E}))=1\).
\end{lemma}

\begin{proof}
    Let us first consider the case \(Q_*=1\). It follows directly from Lemma \ref{Lemma: Subset Lemma} that \(\widehat{\Theta}^{(n)}\) is non-empty. Since every edge \(e\in E_n(\widehat{E})\) is of the form \(e_1\cdot \ldots\cdot e_n\), where \(e_i\in \theta^*\) for all \(1\leq i\leq n\), and since \(\theta^*\) is a simple path with \(\len(\theta^*)=L_*\), it follows that \(\lvert E_n(\widehat{E})\rvert = \len(\theta^*)^n= L_*^n\). We also have that \(\min\{\len(\theta)\mid\theta\in\Theta^{(n)}\}=L_*^n\), and so every path in \(\widehat{\Theta}^{(n)}\) must contain every edge in \(E_n(\widehat{E})\). Consequently, the density \(\rho_n\equiv L_*^{-n}\) is \(\widehat{\Theta}^{(n)}\)-admissible with \(\cM_1(\rho_n)=1\), and therefore 
        \[
            1=\Mod_1(\widehat{\Theta}^{(n)},G_n(\widehat{E}))=\Mod_{Q_*}(\widehat{\Theta}^{(n)},G_n(\widehat{E})).
        \]

    Next, we consider the case \(Q_*>1\). Let $\rho^*_n$ be the optimal density function for $\Mod_{Q_*}(\Theta^{(n)}, G_n)$ and $\cJ^*_n$ the optimal flow for $\Res_p(X_-,X_+,G_n)$. By Lemma \ref{Lemma: Optimal Density} we have that \(\rho^*_n(e_1\cdot\ldots\cdot e_n)>0\) if and only if \(e_i\in \supp(\rho^*_1(e))=\widehat{E}\) for all \(1\leq i\leq n\). Moreover, by Proposition \ref{Proposition: Duality}, we have that \(\lvert \cJ^*_n(e)\rvert >0\) if and only if \(\rho^*_n(e)>0\). Therefore,
        \[
            E_n(\widehat{E})=\{e\in E_n\mid \rho^*_n(e)>0\}=\{e\in E_n\mid \lvert\cJ^*_n(e)\rvert>0\}.
        \]
     The support of a non-trivial flow must contain a path and thus by Lemma \ref{Lemma: Subset Lemma} it follows that \(\widehat{\Theta}^{(n)}\) is non-empty.  It is clear that \(\cJ^*_n\vert_{V_n(\widehat{E}) \times E_n(\widehat{E})}\) is a unit flow and that \(\rho^*_n\vert_{E_n(\widehat{E})}\) is a \(\widehat{\Theta}^{(n)}\)-admissible density. Since
        \[
            \cM_{Q}(\rho^*_n\vert_{E_n(\widehat{E})})=\sum_{e:\rho^*_n(e)\neq 0}\rho^*_n(e)^{Q}=\cM_{Q}(\rho^*_n),
        \]
    and
        \[
            \cE_{Q}(\cJ^*_n\vert_{\cI(G_n(\widehat{E}))})=\sum_{e:\lvert \cJ^*_n(e)\rvert\neq0}\lvert\cJ^*_n(e)\rvert=\cE_{Q}(\cJ^*_n),
        \]
    for all \(1\leq Q<\infty\), it follows that
        \[
            \cM_{Q_*}(\rho^*_n\vert_{E_n(\widehat{E})})^{\frac{1}{Q_*}}\cdot\cE_{\frac{Q_*}{Q_*-1}}(\cJ^*_n\vert_{\cI(G_n(\widehat{E}))})^{\frac{Q_*-1}{Q_*}}=1,
        \]
    and therefore, since \(\widehat{\Theta}^{(n)}\) is non-empty, it follows from Proposition \ref{Proposition: Duality}, that 
        \[
            \Mod_{Q_*}(\widehat{\Theta}^{(n)},G_n(\widehat{E}))= \cM_{Q_*}(\rho^*_n\vert_{E_n(\widehat{E})})=\cM_{Q_*}(\rho^*_n)=\Mod_{Q_*}(\Theta^{(n)},G_n)=1.
        \]
    
\end{proof}

\begin{corollary}\label{Corollary: Density to exponent}
    Let \(q\geq 1\). If for every \(p>q\) there exists a \(\widehat{\Theta}^{(n)}\)-admissible density \(\rho:E_n\to(0,1)\) for some \(n\in\N\), such that \(\cM_p(\rho)<1\), then \(q\geq Q_*\).
\end{corollary}
\begin{proof}
    By Lemma \ref{Lemma: Critical exponent (porous)}, we have that  \(\Mod_{Q_*}(\widehat{\Theta}^{(n)},G_n(\widehat{E}))=1\) for all \(n\in\N\). Therefore, by Lemma \ref{Lemma: Continuity of modulus}, we have for all \(n\in\N\) that 
        \[
            \Mod_q(\widehat{\Theta}^{(n)},G_n(\widehat{E}))>\Mod_p(\widehat{\Theta}^{(n)},G_n(\widehat{E}))>1
        \]
    whenever \(1\leq q<p<Q_*\). Therefore, for any \(n\in\N\) and \(1\leq q<p<Q_*\), there cannot exists a \(\widehat{\Theta}^{(n)}\)-admissible density \(\rho\) for which \(\cM_p(\rho)<1\).
\end{proof}

We are now ready to prove the lower bound for conformal dimension. The following proposition could be proved also with the techniques from \cite{Carrasco}. This would be quite indirect, and we instead prefer to give a short and self-contained proof.

    \begin{proposition}\label{Proposition: Equality of Confdim}
        Let \(Q_*\geq 1\) be the unique value for which \(\Mod_{Q_*}(\Theta^{(1)},G_1)=1\). Then, \(\dim_\AR(\widehat{X},d_{L_*})\geq Q_*\)
    \end{proposition}
    \begin{proof}

        Let \(p,q\in (1,\infty)\) so that \(p>q>\dim_\AR(\widehat{X},d_{L_*}\vert_{\widehat{X}})\), and choose a $q$-Ahlfors regular metric \(\delta\in\cG_\AR(\widehat{X},d_{L_*})\). For every \(n\in\N\), define the density \(\rho_n\colon E_n(\widehat{E}) \to (0,1)\) by
            \[
                \rho_n(e):=\frac{\diam(\widehat{X}_e,\delta)}{\dist(X_-,X_+,\delta)}.
            \]
        Since $X_-\cap X_+=\emptyset$ we have $\dist(X_-,X_+,\delta)>0$. Given \(\theta=[v_0,e_1,\ldots,e_k,v_k]\in\widehat{\Theta}^{(n)}\), we have that \(\widehat{X}_{e_i}\cap \widehat{X}_{e_{i+1}}\neq\varnothing\) for all \(1\leq i\leq k-1\), and that \(\widehat{X}_{e_1}\cap X_-\neq \varnothing\) and \(\widehat{X}_{e_k}\cap X_+\neq \varnothing\). It therefore follows from the triangle inequality that \(\dist(\widehat{X}_-,\widehat{X}_+,\delta)\leq\sum_{e\in \theta}\diam(\widehat{X}_e,\delta)\), and hence
            \[
                L_{\rho_n}(\theta)=\sum_{e\in \theta}\rho_n(e)=\sum_{e\in \theta}\frac{\diam(\widehat{X}_e,\delta)}{\dist(X_-,X_+,\delta)}\geq\frac{\dist(X_-,X_+,\delta)}{\dist(X_-,X_+,\delta)}=1.
            \]
        Thus, the density \(\rho_n\) is \(\widehat{\Theta}^{(n)}\)-admissible for every \(n\in\N\). Now, we wish to estimate \(\cM_p(\rho_n)\). By Lemma \ref{Lemma: Subset Lemma} there exists a constant \(K>1\), such that we may choose for every edge \(e\in \bigcup_{n\in\N}E_n(\widehat{E})\) a point \(x(e)\in \widehat{X}_e\) and a radius \(r(e)>0\) so that 
            \begin{equation}\label{Equation: Cocentric balls}
                B_\delta(x(e),r(e))\subset \widehat{X}_e\subset B_\delta(x(e),Kr(e)).
            \end{equation}
        Denote \(B_e:=B_\delta(x(e),r(e))\). Recall that the measure $\mu=\cH^q$ (with respect to the metric $\delta$) satisfies the condition in \eqref{AReq}. By equation \eqref{Equation: Cocentric balls} and the $q-$Ahlfors regularity of \(\delta\), there exists some constant \(C>1\) such that
            \begin{equation}\label{Equation: Diameter Estimate}
                C^{-1}\cH^q(B_e,\delta)\leq\diam(\widehat{X}_e,\delta)\leq CK^q\cH^q(B_e,\delta)
            \end{equation}
        for every edge \(e\in \bigcup_{n\in\N}E_n(\widehat{E})\). Since \(\mathrm{int}\widehat{X}_e\cap\mathrm{int}\widehat{X}_f=\varnothing\) whenever \(e,f\in E_n(\widehat{E})\) and \(e\neq f\) for some \(n\in\N\), the collection of balls \(\{B_e\}_{e\in E_n(\widehat{E})}\) is disjoint and
            \begin{equation}\label{Equation: Measure Estimate}
                \sum_{e\in E_n(\widehat{E})}\cH^q(B_e,\delta)\leq \cH^q(X,\delta)<\infty.
            \end{equation}
        Let \(M_n:=\max\{\rho_n(e)\mid e\in E_n(\widehat{E})\}\). From the quasisymmetry condition, it follows easily that $M_n \leq \diam(\widehat{X},\delta)\eta(L_*^{1-n})$ and $\lim_{n\to\infty} M_n = 0$. Since \(\rho_n(e)^p\leq\rho_n(e)^{p-q}\rho_n(e)^q\leq M_n^{p-q}\rho_n(e)^q\) for all \(e\in E_n\), we may estimate:
            \begin{equation}\label{Equation: Density Estimate}
            \begin{split}
                \cM_p(\rho_n)&\leq M_n^{p-q}\sum_{e\in E_n(\widehat{E})}\rho_n(e)^q=M_n^{p-q}\sum_{e\in E_n(\widehat{E})}\frac{\diam(\widehat{X}_e,\delta)^q}{\dist(X_-,X_+,\delta)^q}\\
                &\stackrel{\eqref{Equation: Diameter Estimate}}{\leq} M_n^{p-q}\sum_{e\in E_n(\widehat{E})}\frac{CK^q\cH^q(B_e,\delta)}{\dist(X_-,X_+,\delta)^q}\stackrel{\eqref{Equation: Measure Estimate}}{\leq} M_n^{p-q}\frac{CK^q\cH^q(X,\delta)}{\dist(X_-,X_+,\delta)^q}.
            \end{split}    
            \end{equation}
        Since \(p-q>0\) and \(M_n\stackrel{n\to\infty}{\to}0\), it follows that \(M_n^{p-q}\stackrel{n\to\infty}{\to}0\), and so there exists some \(n\in\N\) such that \(\cM_p(\rho_n)<1\). Since, \(p>q\) was arbitrary, it follows from Corollary \ref{Corollary: Density to exponent} that \(q\geq Q_*\). Since \(q>\dim_\AR(\widehat{X},d_{L_*}\vert_{\widehat{X}})\) was arbitrary, it follows that \(\dim_\AR(\widehat{X},d_{L_*})\geq Q_*\).
    \end{proof}

        We are now ready to prove the main Theorems.
    
    \begin{proof}[Proof of Theorem \ref{Main Theorem: Value of Conformal Dimension}]
        It follows from Lemmas \ref{Lemma: Critical exponent (porous)} and \ref{Lemma: Subset Lemma}, that the subset \(\widehat{X}\) is complete and uniformly perfect. Thus, as \(\dim_\AR(X,d_{L_*})\leq \log_{L_*}(\lvert E_1\rvert)<\infty\), it follows Corollary \ref{Corollary: Confdim bound for uniformly perfect subset} that
            \begin{equation}\label{Equation: Confdim inequality}
                \dim_\AR(\widehat{X},d_{L_*})\leq\dim_\AR(X,d_{L_*}).
            \end{equation}
        By Corollary \ref{Corollary: Confdim upper bound} we have that \(\dim_\AR(X,d_{L_*})\leq Q_*\) and by Proposition \ref{Proposition: Equality of Confdim} that \(\dim_\AR(\widehat{X},d_{L_*})\geq Q_*\). Therefore it follows that 
            \[
                Q_*=\dim_\AR(X,d_{L_*})=\dim_\AR(\widehat{X},d_{L_*}).
            \]
    \end{proof}

    \begin{proof}[Proof of Theorem \ref{Main Theorem: Characterization of Attainment}]
        If the IGS has removable edges, meaning that \(E_1\neq \widehat{E}_1\), then Lemma \ref{Lemma: Porous} implies that the subset \(\widehat{X}\) is porous. As shown in the proof of Theorem \ref{Main Theorem: Value of Conformal Dimension}, \(\dim_\AR(X,d_{L_*})=\dim_\AR(\widehat{X},d_{L_*})\) and hence, by Proposition \ref{Proposition: Porosity implies non-attainment}, the conformal dimension of \((X,d_{L_*})\) is not attained. On the other hand, if the IGS does not have removable edges, then, by definition, there exists a symmetric \(\Theta^{(1)}\)-admissible density \(\rho^*\colon E_1\to (0,1)\) so that \(\cM_{Q_*}(\rho^*)=1\). It then follows from Theorem \ref{Main Theorem: Construction of Metrics} that 
        \(d_{\rho^*}\in \cG_\AR(X,d_{L_*})\) and
            \[
                \dim_\mathrm{H}(X,d_{\rho^*})=Q_*=\dim_\AR(X,d_{L_*}),
            \]
        where the last equality is given by Theorem \ref{Main Theorem: Value of Conformal Dimension}. Thus, the conformal dimension of \((X,d_{L_*})\) is attained.
    \end{proof}
    
\section{On the Combinatorial Loewner Property}\label{Section: Approximate self-similarity and the Combinatorial Loewner Property}
One of the motivations for this work comes from the combinatorial Loewner property (CLP) discussed in \cite{BourK}, and the related \emph{Kleiner's conjecture} \cite[Conjecture 7.5]{KleinerICM}. The conjecture states that the combination of CLP and approximate self-similarity implies that a space is quasisymmetric to a Loewner space. By \cite[\S 1.6]{CEB} (or \cite[Proposition 11.14]{murugan2023first}), this problem is equivalent to the attainment of conformal dimension for such spaces. In \cite{anttila2024constructions}, the first two authors found counterexamples to this conjecture. All of these examples came from IGSs with removable edges, and the reason for non-attainment was the same as in Theorem \ref{Main Theorem: Characterization of Attainment}. While these counterexamples resolved the conjecture for specific spaces, it remains a difficult open problem to understand the general case; see for example the discussion in \cite{davidsebsphere}. In particular, it would be important to understand if the conjecture has an affirmative answer for some classes of spaces. Which additional assumptions are needed to guarantee that a space is quasisymmetric to a Loewner space?

In this section, we find a large class of non-trivial cases where Kleiner's conjecture is true. This is done by verifying the approximate self-similarity and the combinatorial Loewner properties for the limit spaces $X$ that we have constructed. Thus, in the absence of removable edges, Theorem \ref{Main Theorem: Characterization of Attainment} implies attainment of the conformal dimension and hence verifies Kleiner's conjecture for many symmetric Laakso-type spaces. As seen in earlier sections, the metrics attaining the conformal dimension are explicit, but do not arise from a snowflake deformation, or some other trivial construction.

An additional motivation for this section is to observe that symmetric Laakso-type spaces of minimal Hausdorff dimension are \emph{Loewner spaces}. This follows immediately from \cite[\S 1.6]{Ch} (or \cite[Proposition 11.14]{murugan2023first}), and is stated at the end of the section. For the definition of a Loewner space, see \cite[Definition 1.11]{Ch} or \cite{HK}.

In order discuss Kleiner's conjecture, we will need to establish approximate self-similarity and CLP for the spaces of interest. Most of the work on verifying CLP was already done in  \cite{anttila2024constructions}, and the only difference here is that we used a slightly different metric. For expository purposes, we define approximate self-similarity and CLP, and outline the general approach to prove these for our metrics. For most of the technical steps we refer to \cite{anttila2024constructions}, and we simply highlight any changes.

\begin{definition}\label{def:approxselfsim}
    A metric space $(X,d)$ is approximately self-similar if there exists a constant $L>1$ s.t. for every $x\in X, r\in (0,\diam(X)]$ there exists an open set $U_{x,r}\subset X$ and a $L$-bi-Lipschitz map $f:(B(x,r),d/r)\to U_{x,r}$. 
\end{definition}

We recall the definition of the intrinsic metric $d_{L_*}$ from Definition \ref{def:intrinsicmetric}.

\begin{proposition} A symmetric Laakso-type fractal space \((X,d_{L_*})\) is approximately self-similar.
\end{proposition}
\begin{proof}
    The proof is very similar to \cite[Section 3.8]{anttila2024constructions} and we simply sketch the details. For every $e\in G_n$, let $\cN(e)$ be the subgraph of size $1$ about $e$ and $\widehat{\cN}(e)$ a subgraph of size two. We say that two edges $e\in G_n, f\in G_m$ are combinatorially equivalent, if $\widehat{\cN}(e), \widehat{\cN}(f)$ are isomorphic as graphs via an isomorphism $\psi$, which satisfies $\phi_{\psi(u),\psi(a)}=\phi_{u,a}$ for each edge $u$ in $\cN(e)$ and $a\in u$. 

    Further, any edge path in the definition of the metric connecting $y,z\in \cN(e)$ which is shortest must be contained in $\widehat{\cN(e)}$. Using these facts, and the proof of \cite[Corollary 3.34]{anttila2024constructions} one sees that there is a $L_{*}^{n-m}$-homothety $\cN(e)\cdot X\to \cN(f)\cdot X$, where $A\cdot X = \bigcup_{e\in A} e\cdot A$ for any $A\subset \bigcup_{m=1}^\infty E_m$.

    Finally, approximate self-similarity follows from finiteness of isomorphism classes and these homotheties. 
    It is direct to see that there are finitely many isomorphism classes of edges, and there is some $N\in \N$ so that each isomorphism class already contains a representative for some $e\in G_n$ with $n\leq N$. Let $x\in e\cdot X$ and $r<L_*^{-N}$. By Property \ref{Lemma: DL(4)}, we have $B(x,r) \subset \cN(e)\cdot X$ for some edge $e\in G_m$ with $L_*^{-m-1}\leq r\leq L_*^{-m}$.  Thus, the previous paragraph constructs the desired homothety. This suffices to show approximate self-similarity; see \cite[Proposition 3.35.]{anttila2024constructions} for further details.
\end{proof}

The Combinatorial Loewner Property (CLP) is more technical to state, and we refer to \cite{anttila2024constructions} for the classical definition from \cite{BourK} and \cite{clais}. Here, we will present a simpler definition better suited for our setting.

For $v\in V_n$, let $X_v = \bigcup_{e\in E_n, v\in e} e \cdot X$. This set resembles a star around $v$, and $X_v\cap X_w$ if and only if $v,w$ are connected by an edge. If $A,B\subset X$ are two subsets, let $\Gamma(A,B)$ be the collection of continuous paths connecting $A$ to $B$, and $\Gamma_L(A,B)$ the collection of continuous curves with diameter at most $L$. A function $\rho:V_n\to [0,1]$ is called \emph{admissible} for a collection of curves $\Gamma$ if 
\[
\sum_{v\in V_n, X_v\cap \gamma \neq \varnothing} \rho(v)\geq 1
\text{ for all } 
\gamma \in \Gamma.
\]
Let
\[
\Mod_p^D(\Gamma;G_n)=\inf \sum_{v\in V_n} \rho(v)^p,
\]
where the infimum is taken over all admissible functions $\rho$ for $\Gamma(E,F)$.

\begin{definition} Let $Q>1$. We say that $(X,d_{L_*})$ satisfies the \emph{$Q$-Combinatorial Loewner property} if the following hold.
\begin{enumerate}[label={\color{blue}{\textup{$($CLP\arabic*}$)$}}, widest=a, leftmargin=*]
\item\label{CLP1} For every $A\geq 4$ there exists a $\delta>0$ and $L\geq 1$ s.t. for all $x,y\in X, r>0$ with $d_{L_*}(x,y)\leq Ar$, and all $n>0$ s.t. $L_*^{-n}<r/2$,
\[
\Mod_Q^D(\Gamma_{Lr}(B(x,r),B(y,r); G_n))\geq \delta.
\]
\item\label{CLP2} For every $A\geq 2$ there exists a $\delta_A$ with $\lim_{A\to \infty} \delta_A = 0$ st. the following holds. For every $x\in X$, $r<\diam(X)=1/2$ and $n>0$ s.t. $L_*^{-n}<r/2$, we have
\[
\Mod_Q^D(\Gamma(B(x,r),X \setminus B(x,Ar)); G_n)\leq \delta_A.
\]
\end{enumerate}
\end{definition}
This definition is different from the classical one, but can easily be seen to be equivalent. Unlike the classical definition, here we do not use the notion of an $\alpha$-approximation and additionally \ref{CLP1} only gives a lower bound for moduli of curves connecting balls instead of all pairs of sets; see \cite{anttila2024constructions} for the classical definition. The equivalence is seen by combining two observations. First, one can reduce from all sets to balls by \cite[Proposition 2.9]{BourK}. Second, our graphs $G_n$ correspond isomorphically with an $\alpha$-approximation. This follows by an argument in \cite[Proposition 3.37]{anttila2024constructions} combined with Lemma \ref{Lemma: Subset Lemma} and Lemma \ref{Lemma: Distance Lemma 1} -- which account for the fact that we are using a slightly different metric. This argument shows that $G_n$ is the incidence graph of the covering $\{X_v : v\in V_n\}$, and that this covering forms an $\alpha$-approximation.

So far, in this paper we have studied the edge modulus, and fortunately this is comparable to the corresponding vertex modulus. If $A \subset X$ is a set, let $G_n(A)=\{v\in V_n : X_v \cap A\neq \varnothing\}$. 
\begin{lemma}[{\cite[Proposition 5.4.9]{anttila2024constructions}}]\label{lem:dmodemod} There exists a constant $C\geq 1$ s.t. the following holds. If $A,B\subset X$ s.t. $G_n(A),G_n(B)$ are disjoint, then
    \[
    C^{-1}\Mod_p^D(A,B;G_n) \leq\Mod_p(\Theta(G_n(A),G_n(B));G_n) \leq C\Mod_p^D(A,B;G_n),
    \]
    where the comparability constant is independent of $n$.
\end{lemma}

Further, a simple argument allows us to compute moduli at different levels of the graph. Let $M_Q=\Mod_Q(\Theta^{(1)},G_1)$. 

\begin{lemma}
[{\cite[Theorem 4.38]{anttila2024constructions}}]
\label{lem:modulusbound} Let $n>m.$ For every $A,B\subset V_m$ we have
\[
\Mod_p(\Theta(\pi_{n,m}^{-1}(A),\pi_{n,m}^{-1}(B)), G_n) = \Mod_p(\Theta(A,B),G_m) M_Q^{m-n}.
\]
\end{lemma}

We can now prove CLP for the limits space $X$ by adapting \cite[Theorem 5.2]{anttila2024constructions}. Notice that the value $Q>1$ s.t. a metric space is $Q$-CLP is unique. Indeed, if there is any such $Q$, then $Q$ is equal to the conformal dimension; see e.g. \cite[Lemma 4.2]{EBConf}. 

\begin{theorem}\label{thm:CLP} 
A symmetric Laakso-type fractal space \((X,d_{L_*})\) is \(Q\)-CLP for some \(Q>1\) if and only if the associated family of paths \(\Theta^{(1)}\) contains two edge-disjoint paths.
\end{theorem}

\begin{proof}

The necessity of the condition is shown by contrapositive. 
If $G_1$ does not have two edge disjoint paths, then by Lemma \ref{Lemma: Cut-points}, we have that $\Mod_1(\Theta^{(1)},G_1)=1$ and by Theorem \ref{Main Theorem: Value of Conformal Dimension}, we have that the Ahlfors regular conformal dimension is $1$. On the other hand a $Q$-CLP space with $Q>1$ has $\dim_{AR}(X)=Q>1$ by \cite[Lemma 4.2]{EBConf}. Thus, in this case the limit space can not be $Q$-CLP.

Next, we show sufficiency. Let $Q=\dim_{\AR}(X)$, which by Lemma \ref{Lemma: Cut-points}, Proposition \ref{Proposition: Equality of Confdim} and Corollary \ref{Corollary: Confdim upper bound} satisfies $Q>1$. The bound \ref{CLP2} follows from the argument in \cite[Proposition 5.2]{anttila2024constructions}, where the relevant bound $\diam(X_v)\leq 2L_*^{-n}$ follows from \ref{Lemma: DL(6)}.

Next, the bound \ref{CLP1} follows from \cite[Proposition 5.7]{anttila2024constructions}. Here, one needs to use Lemma \ref{Lemma: Subset Lemma} to replace the use of \cite[Proposition 3.37]{anttila2024constructions}.
\end{proof}

Finally, from the CLP we also obtain the classical Loewner property for minimizers. See \cite{HK} for the classical definition of a Loewner space, and \cite{CEB} for more discussion on this notion.

\begin{corollary}\label{cor:Loewner}
    A symmetric Laakso-type fractal $(X,d_{L_*})$ is quasisymmetric to a $Q$-Ahlfors regular $Q$-Loewner space for $Q>1$ if and only if the associated IGS doesn't have a removable edge and $Q=\dim_{\AR}(X)>1$. Moreover, if  $(X,d)$ is any minimizer for the Ahlfors regular conformal dimension and $Q=\dim_{\AR}(X)>1$, then $(X,d)$ is $Q$-Loewner. 
\end{corollary}
\begin{proof}
If $(X,d_{L_*})$ is quasisymmetric to a $Q$-Ahlfors regular $Q$-Loewner space for some $Q>1$, then by \cite[1.7. Corollary]{TQuasi}, $Q$ must be equal to the Ahlfors regular conformal dimension of $(X,d_{L_*})$. Thus, $X$ attains its Ahlfors regular conformal dimension $Q$ and by Theorem \ref{Main Theorem: Characterization of Attainment} there cannot be any removable edge.

On the other hand, if $(X,d_{L_*})$ does not have a removable edge and $\dim_{\AR}(X)>1$, then by Theorem \ref{Main Theorem: Characterization of Attainment} it attains its Ahlfors regular conformal dimension and $Q=\dim_{\AR}(X)>1$. By Theorem \ref{thm:CLP} $(X,d_{L_*})$ is also $Q$-CLP, and by \cite[\S 1.6]{CEB} any minimizer of Ahlfors regular conformal dimension is $Q$-Loewner. This completes the proof of the converse statement and also the moreover-clause.

\end{proof}

\bibliographystyle{acm}
\bibliography{clp}

\appendix

\end{document}